\documentclass[sn-mathphys,Numbered]{sn-jnl} 

 \usepackage[T1]{fontenc}
\usepackage{mathtools}
\usepackage[thinc]{esdiff}
\usepackage{wrapfig,lipsum}
\usepackage{anyfontsize}
\usepackage{graphicx} 
\usepackage{multirow} 
\usepackage{amsmath,amssymb,amsfonts} 
\usepackage{amsthm} 
\usepackage{mathrsfs} 
\usepackage[title]{appendix} 
\usepackage{xcolor} 
\usepackage{float}
\usepackage{textcomp} 
\usepackage{manyfoot} 
\usepackage{booktabs} 
\usepackage{algorithm} 
\usepackage{algorithmicx} 
\usepackage{algpseudocode} 
\usepackage{listings}
\usepackage{array}
\usepackage{pdflscape}
\usepackage{caption}
\usepackage{subcaption}
\usepackage[utf8]{inputenc}
\usepackage{booktabs}
\usepackage{amsfonts}
\usepackage{siunitx}
\usepackage{longtable}
\usepackage{eucal}
\usepackage{multirow}
\usepackage{enumerate}
\usepackage{lineno}
\usepackage{framed}

\usepackage{tikz}
\usetikzlibrary{arrows} 
\usepackage{pgfplots} 
\usepackage{hyperref}

\usepackage{placeins}
\usepackage[disable]{todonotes}
\newtheorem{theorem}{Theorem}[section]
\newtheorem*{tpp}{Three Point Property}
\newtheorem{proposition}{Proposition}[section] 
\newtheorem{lemma}{Lemma}[section]

\newtheorem{remark}{Remark}[section] 
\newtheorem{assumption}{Assumption}[section] 
\newtheorem{definition}{Definition}[section]

\begin{document}
\title[Proximal gradient methods for Unconstrained set optimization]{Proximal Gradient Methods for Unconstrained Set Optimization Problems with Set-Valued Maps of Finite Cardinality}

\author[1]{\fnm{Ravi} \sur{Raushan}}\email{raviraushan.rs.mat21@itbhu.ac.in}

\author*[1]{\fnm{Debdas} \sur{Ghosh}}\email{debdas.mat@iitbhu.ac.in}

\author[1]{\fnm{Anshika}}\email{anshika.rs.mat19@itbhu.ac.in}

\author[2]{\fnm{V.} \sur{Vetrivel}}\email{vetri@iitm.ac.in} 

\affil[1]{\orgdiv{Department of Mathematical Sciences}, \orgname{Indian Institute of Technology (BHU)}, \orgaddress{\city{Varanasi}, \postcode{221005}, \state{Uttar Pradesh}, \country{India}}}

\affil[2]{\orgdiv{Department of Mathematics}, \orgname{Indian Institute of Technology Madras}, \orgaddress{\city{Chennai}, \postcode{600036}, \state{Tamil Nadu}, \country{India}}}

\abstract{This work presents two different types of proximal gradient methods, with line search and without line search, for solving unconstrained set-valued optimization problems under the lower set-less ordering relation induced by a solid cone that is convex, pointed, and closed. The objective mapping of the problem involves finitely many functions, with each one being the sum of a continuously differentiable function and a convex function that is proper and closed. We present an approach to characterize weakly minimal points of the problem with the help of weakly efficient points of a family of vector optimization problems. Thereafter, we establish a stationarity condition along with its connection with weakly minimal points of the problem under study. Based on the stationary condition, the concept of a descent direction at a non-stationary point is discussed. In view of the line search-based method, we formulate an Armijo-type line search condition and establish the existence of such a step-size. For the proposed methods, global convergence is established under mild assumptions. The convergence analysis of the proximal gradient method with line search provides a theoretical advancement over the convergence results previously established for the steepest descent method in set-valued optimization problems. In addition, we analyze the computational complexity of the proposed methods and show that both methods achieve a convergence rate of $\mathcal{O}(1/\sqrt{k})$. Numerical results are reported to test the performance of the methods in practice. \\


\smallskip


}

\maketitle
\section{Introduction}
Set optimization addresses the problem of optimizing a set-valued map whose codomain set is partially ordered by a convex cone that is closed and pointed. It generalizes the conventional scalar or vector optimization framework by allowing the objective to be a set-valued map rather than a real-valued function or a vector-valued function.
Thus, set optimization accommodates a more complex objective map, which makes its study particularly effective in scenarios involving uncertainty, conflicting goals, or incomplete preference information. 
Due to its modeling flexibility, set optimization has profound applications across a wide range of disciplines, including variational inequalities, duality theory, robust optimization, image processing, viability theory, mathematical economics, etc.\  \cite{khan2016set}. 

There are two widely used approaches to solve set optimization problems: the \emph{vector approach} and the \emph{set approach}. A key disadvantage of the vector approach is its inability to accommodate situations where the decision maker evaluates preferences based on comparisons between elements of the image set. To overcome this limitation, the set approach has been adopted. For more clarity, see \cite[Example 2.6.21]{khan2016set}.
In the set approach, solution concepts are formulated through comparing the images of the objective mappings, with a focus on identifying minimal (or maximal) solutions of the problem. This framework was first introduced systematically by Kuroiwa \cite{Kuroiwa1996min} in the context of set optimization and has since been widely adopted in the literature (see, \cite{steepest2021set_optimization,ghosh2025quasi, ghosh2024newton, Ghosh CGM}). In this article, we adopt a set approach as formulated in \cite{Kuroiwa1996min}, employing the lower set-less pre-ordering ($\preceq^{\ell}$) and strict ordering ($\prec^{\ell}$) relations associated with a cone $C \subset \mathbb{R}^{m}$ defined as follows: for $U,V\subseteq \mathbb{R}^{m}$,
\begin{equation}\label{order_relation}
    U \preceq^{\ell} V \iff V \subseteq U + C \quad \text{and}\quad
    U \prec^{\ell} V \iff V \subseteq U + \mathrm{int}(C),    
\end{equation}
where the cone $C$ is assumed to be pointed ($C \cap (-C) = \{0\}$), closed, convex, and solid ($\mathrm{int}(C) \neq \emptyset$). This work aims to solve the following set-valued optimization problem:
\begin{equation}\tag{$\text{SOP}_{\ell}$}\label{sop}
    \preceq^{\ell}\text{-}\min H(x),
\end{equation}
where the set-valued map $H: \mathbb{R}^{n}\rightrightarrows\mathbb{R}^{m}\cup\{+\infty_{C}\}$ is defined by 
\[
H(x):=\{h^{1}(x), h^{2}(x),\ldots,h^{p}(x)\},
\]
such that $h^{j}:= f^{j}+g^{j}$ for each $j=1,2,\ldots,p$; here $+\infty_{C}$ represents the element of $\mathbb{R}^{m}$ plane such that $w\prec_{C} +\infty_{C}$ for all $w\in \mathbb{R}^{m}$. For any index $j\in \{1,2,\ldots,p\}$, the function $f^{j}:\mathbb{R}^{n}\rightarrow \mathbb{R}^{m}$ is assumed to be continuously differentiable and $g^{j}:\mathbb{R}^{n}\rightarrow \mathbb{R}^{m}\cup\{+\infty_{C}\}$ is a convex function which is proper and closed.

Several methodological frameworks have been developed to solve set optimization problems, each addressing specific structural properties of the objective map. These include scalarization-based techniques, sorting-type procedures, derivative-free methods, and first-order algorithms; see \cite{steepest2021set_optimization}. Notably, when derivative information of the functions $h^{j}$, $j=1,2,\ldots,p$, involved in the objective map $H$ is available, first-order methods are typically more effective and computationally preferable to the other approaches. Accordingly, this work introduces methods for solving problem~\eqref{sop}, based on the proximal gradient technique.

When $p=1$, the problem~\eqref{sop} reduces to the following vector optimization problem:
\begin{equation}\tag{$\text{VOP}_{C}$}\label{VOP}
\underset{x\in \mathbb{R}^{n}}{\min} h(x),
\end{equation}
where the function $h: \mathbb{R}^{n} \rightarrow \mathbb{R}^{m}\cup\{+\infty_{C}\}$ is defined as $h(x) := f(x) + g(x)$. Here, $f: \mathbb{R}^{n} \rightarrow \mathbb{R}^{m}$ is a continuously differentiable function and $g: \mathbb{R}^{n} \rightarrow \mathbb{R}^{m}\cup\{+\infty_{C}\}$ is a proper closed convex function, which may not be differentiable.

For solving the multi-objective version of the problem~\eqref{VOP} (with $C$ is the non-negative orthant of $ \mathbb{R}^{m}$), the proximal gradient method is mostly studied as an iterative approach, particularly for large-scale problems. In~\cite{tanabe2019proximal}, two variants were proposed---one with line search and one without line search---the latter assuming Lipschitz continuity on the gradient of the smooth term $f$ of $h$. Under the same assumption, convergence of the without line search variant was later analyzed in~\cite{tanabe2023convergence}. Based on~\cite{tanabe2019proximal}, many other variants have been developed; see \cite{chen2023barzilai,tanabe2023accelerated, zhao2025proximal} and the references therein.
Beyond these classical schemes, a Bregman proximal gradient method was introduced in~\cite{chen2024bregman}, where the Euclidean distance is used in place of a Bregman distance, allowing convergence analysis under relative smoothness assumptions.
In order to solve the problem~\eqref{VOP} with a generic cone $C$, comparatively fewer methods are available. An inertial forward--backward method for computing weakly efficient solutions under convexity was proposed in~\cite{bot2018inertial}. A proximal gradient splitting scheme that relaxes the Lipschitz continuity assumption on the gradient was introduced in~\cite{bello2022proximal}, where convergence to weakly efficient solutions was established.

To our understanding, no method has been established so far for solving the non-smooth composite optimization problem \eqref{sop}. In the special case where $g^{j} = 0$, various methods have been introduced in the literature. The earliest work in this direction is by Bouza et al. \cite{steepest2021set_optimization}, who proposed a steepest descent method with an Armijo-type line search condition, based on a first-order solution framework. 
Subsequently, quasi-Newton and Newton-type methods \cite{ghosh2025quasi, ghosh2024newton} were developed on the basis of this framework, under appropriate assumptions on the objective map, and their convergence rates were analyzed. 
A limitation of these methods \cite{steepest2021set_optimization, ghosh2025quasi, ghosh2024newton} is that their convergence analysis requires a regularity \cite[Definition 4.1]{steepest2021set_optimization} assumption at the solution point. 
More recently, Ghosh et al. \cite{Ghosh CGM} proposed conjugate gradient methods with Wolfe-type line search conditions, also within a first-order solution framework, and established convergence results without imposing any regularity assumptions.

The main motivations and contributions of this work are provided below:
\begin{itemize}
    \item Most of the existing first-order methods for~\eqref{sop} require regularity assumptions on the solution set to guarantee convergence. In contrast, in the proposed methods, we establish convergence without such assumptions. In particular, we show that the steepest descent method proposed in~\cite{steepest2021set_optimization} converges without imposing regularity conditions. 

    \item Computational complexity is a key factor in analyzing optimization methods. To our understanding, no complexity analysis has been provided for existing approaches to~\eqref{sop}. Here, we present such an analysis for the proposed methods.

    \item The problem~\eqref{sop} has several important applications. For instance, an uncertain set-valued optimization problem can be formulated within the framework of~\eqref{sop}. Moreover, certain image processing problems can also be addressed using this formulation. In this study, we show through examples that the proposed methods can be employed directly on uncertain set-valued optimization problems.

    \item For the problem~\eqref{VOP}, there is no proximal gradient method in the literature. Since, for $p=1$, the problem~\eqref{sop} reduces to~\eqref{VOP}, the methods developed in this work are also applicable to solve the vector optimization problem~\eqref{VOP}. 
\end{itemize}

In this article, we neatly propose two proximal gradient methods based on a first-order solution approach: one incorporating a line search condition and the other without it. 

The structure of the article is given as follows. To facilitate the development of the proposed results, Section~\ref{sec.2}  introduces the necessary notations and preliminary concepts. Section~\ref{sec.3}  is devoted to the discussion of optimality conditions and descent directions. We first derive the optimality conditions for the vector optimization problem~\eqref{VOP} and then extend these results to the set optimization problem~\eqref{sop}, which form the foundation of the proposed methods. Based on the optimality conditions of the problem~\eqref{sop}, we introduce the notion of descent directions. Section~\ref{sec.4} details the proposed methods, structured in Algorithm~\ref{algo 1} and Algorithm~\ref{algo}. Their convergence properties are established in Section~\ref{sec.5}, where we also analyze the computational complexity and derive the rate of convergence. In Section~\ref{sec.6}, we illustrate the practical performance of the proposed methods through some~\eqref{sop} test problems. Finally, Section~\ref{sec.7} concludes the study and outlines possible directions for future research.

\section{Preliminaries and terminologies}\label{sec.2}
The notations and terms employed in the sequel are as follows.

\begin{itemize}
    \item $\mathbb{R}$ and $\mathbb{N}$ denote the sets of real numbers and positive integers, respectively.
    \item $\mathbb{R}^m_{+}$ refers to the non-negative orthant of $\mathbb{R}^m$, while $\mathbb{R}^m_{++}$ denotes its positive orthant.
    \item $[p]:=\{1, 2, \ldots, p\}$ for any $[p]\in \mathbb{N}$.
    
    \item $\bar{y}$ is taken as a reference point of $\mathbb{R}^{n}$.
    
    \item Iterative sequence of points of $\mathbb{R}^{n}$ is denoted by $\left\{x^{k}\right\}$.
    \item For any set $A$, the symbols $\mathrm{int}(A)$, $\mathrm{conv}(A)$, and $\mathcal{P}(A)$ represent its interior, convex hull, and power set, respectively.
     \item $\mathrm{dom}(H):=\left\{x\in \mathbb{R}^{n}~\middle |~h^{j}(x)\prec_{C}+\infty_{C} \quad \forall j\in [p]\right\}.$
    \item $C^{*}:=\left\{z\in \mathbb{R}^{m}~\middle |~ \bar{w}^{\top}z\geq 0 ~ \forall  \bar{w}\in C\right\}$ is the dual cone of $C$.
\end{itemize}

Note that cone $C$ for the problem \eqref{sop} is closed and convex. Therefore, $C^{**}=C$. Accordingly, we get
\begin{align*}
     -C=\left\{\bar{w}\in \mathbb{R}^{m}~\middle |~ z^{\top}\bar{w}\leq 0 ~\forall z\in C^{*}\right\} \quad \text{and }
      -\mathrm{int}(C)=\left\{\bar{w}\in \mathbb{R}^{m}~\middle |~ z^{\top}\bar{w}< 0 ~ \forall z\in C^{*}-\{0\}\right\}.
\end{align*}

A convex compact set $\mathcal{Z} \subset \mathbb{R}^m$ is called a generator of the dual cone $C^*$ if the following relation holds:
\[
\text{cone}(\mathrm{conv}(\mathcal{Z})) = C^*.
\]
For analytical purposes, the set $\mathcal{Z}$ is often taken as
\begin{equation}\label{gen_set}
  \mathcal{Z} := \left\{ z \in C^* ~\middle|~ \|z\| = 1 \right\}.  
\end{equation}
Moreover, if the cone $C$ is polyhedral, then so is its dual cone $C^*$, and the set $\mathcal{Z}$ corresponds to the finite collection of extreme rays of $C^*$.

To compare a pair of vectors $\bar{w}$ and $ \hat{w}$ in $ \mathbb{R}^m $, we define partial and strict order induced by the cone \( C \subset \mathbb{R}^m \) as follows \cite{Gopfert_book}:
\[
\bar{w} \preceq_C \hat{w}  \iff  \hat{w} - \bar{w} \in C \quad \text{and} \quad \bar{w} \prec_C \hat{w}  \iff  \hat{w} - \bar{w} \in \mathrm{int}(C).
\]

To scalarize vectors, we use the \textit{Drummond-Svaiter scalarizing function} \( \phi : \mathbb{R}^m \rightarrow \mathbb{R}\cup\{+\infty\} \), defined as follows \cite{drummond2005steepest}: for any \( w \in \mathbb{R}^m \),
\begin{equation*}
    \phi(w) := \sup_{z \in \mathcal{Z}} z^{\top} w,
\end{equation*}
where \( \mathcal{Z} \subset \mathbb{R}^m \) is the generator set as given in \eqref{gen_set}. Note that $\mathcal{Z}$ is compact. Therefore, $\phi$ is well-defined.\\

\begin{lemma} \textnormal{\cite{drummond2005steepest}} \label{scal_func}
Let the function \( \phi \) be defined as in \eqref{gen_set}. Then, for all \(\bar{w}, \hat{w}\in \mathbb{R}^m \) and $c\geq 0$, the following properties hold. 
\begin{enumerate}[(i)]
    \item Sub-linearity:
    \[
    \phi(\bar{w}+ \hat{w}) \leq \phi(\bar{w}) + \phi(\hat{w}), \quad \text{and} \quad \phi(c\,\bar{w}) =c\, \phi(\bar{w}).
    \]
    
    \item Monotonicity:
    \[
    \text{If } \bar{w} \preceq_C \hat{w}, \text{ then } \phi(\bar{w}) \leq \phi(\hat{w}); \quad \text{if } \bar{w} \prec_C \hat{w}, \text{ then } \phi(\bar{w}) < \phi(\hat{w}).
    \]
    
    \item Lipschitz continuity:
    \[
    \phi(\bar{w}) - \phi(\hat{w}) \leq \|\bar{w} - \hat{w}\|.
    \]
    
    \item Representability of the cone:
    \[
    -C = \left\{ \bar{w} \in \mathbb{R}^m \,\middle|\, \phi(\bar{w}) \leq 0 \right\} \quad \text{and}\quad
    -\mathrm{int}(C) = \left\{ \bar{w} \in \mathbb{R}^m \,\middle|\, \phi(\bar{w}) < 0 \right\}.
    \]
\end{enumerate}
\end{lemma}

We next discuss convexity of a function relative to the cone \( C \).
A function $h : \mathbb{R}^n \rightarrow \mathbb{R}^{m}\cup\{+\infty_{C}\}$ is called $C$-convex on a convex set $D \subseteq \mathbb{R}^n$ if, for any two points $\bar{y}, \hat{y} \in D$ and $\lambda \in (0,1]$, one holds:
\[
h(\lambda \bar{y} + (1 - \lambda)\hat{y}) \preceq_C \lambda h(\bar{y}) + (1 - \lambda)h(\hat{y}).
\]
Moreover, if $h$ is differentiable, then $h$ is convex on $D$ if and only if
\[
\phi\left( h(\hat{y}) - h(\bar{y}) - \nabla h(\bar{y})^\top (\bar{y} - \hat{y}) \right) \geq 0 
\quad \forall \bar{y}, \hat{y} \in D.
\]
    


We use the notion that $D_{v}h(\bar{y})$ is the directional derivative of a function $h$ at $\bar{y}$ in the direction of $v\in \mathbb{R}^{n}$, which is defined as
\[D_{v}h(\bar{y}):= \lim_{\alpha\to 0} \frac{1}{\alpha}\left(h(\bar{y}+\alpha v)-h(\bar{y})\right).\]

We now define a function $\Tilde{g}: (0, \infty)\to \mathbb{R}^{m}\cup\{+\infty_{C}\}$ based on a convex function $g:\mathbb{R}^n \rightarrow \mathbb{R}^{m}\cup\{+\infty_{C}\}$ corresponding to points $\bar{y}, ~\tilde{v}\in \mathbb{R}^{n}$ as
\[\Tilde{g}(\alpha):= \frac{1}{\alpha}\left(g(\bar{y}+\alpha \tilde{v})-g(\bar{y})\right).\]
Then, $\Tilde{g}$ has a monotone property. That is, for any $0< \alpha\leq \beta$, the relation $\Tilde{g}(\alpha)\preceq_{C} \Tilde{g} (\beta)$ holds. In particular, for $\beta =1$, we get 
\begin{equation}\label{convex_direction_deri}
    \frac{1}{\alpha}\left(g(\bar{y}+\alpha \tilde{v})-g(\bar{y})\right)\preceq_{C} g(\bar{y}+\tilde{v})-g (\bar{y}) \quad \text{for all } \alpha \in (0, 1]. 
\end{equation}

Next, we present a descent lemma in vector form, adapted from the scalar descent lemma in \cite[Proposition A.24]{Bertsekas_D}, which will be essential to the subsequent analysis.
\\

\begin{lemma}\label{descent lemma}
    Let $f:\mathbb{R}^n \rightarrow \mathbb{R}^m$ be continuously differentiable and $\nabla f$ satisfy an $L$-Lipschitz condition. Then,
    \[\phi\left(f(x+y)-f(x)-\nabla f(x)^{\top}y\right)\leq \frac{L}{2}\|y\|^2 \quad \text{for all } x,y\in \mathbb{R}^{n}.\]
\end{lemma}
\begin{proof}
    Define a function $ \Psi : \mathbb{R}\rightarrow \mathbb{R}^m$ by
    $\Psi (t):= f(x+ty)$.\\ Note that
   $ \diff{\Psi(t)}{t}=\nabla f(x+ty)^{\top}y$. Therefore,
   \begin{align*}
      & f(x+y)-f(x)= \Psi(1)-\Psi(0)= \int_{0}^1 \nabla f(x+ty)^{\top}y\,\mathrm{d}t \\
     \implies  & f(x+y)-f(x) - \nabla f(x)^{\top}y= \int_{0}^1 (\nabla f(x+ty)-\nabla f(x))^{\top}y\,\mathrm{d}t \\
     \implies  & \phi\left(f(x+y)-f(x) - \nabla f(x)^{\top}y\right)= \underset{z\in \mathcal{Z}}{\sup} \left \{ z^{\top}  \int_{0}^1 (\nabla f(x+ty)-\nabla f(x))^{\top}y\,\mathrm{d}t \right\},   
   \end{align*}
   by the definition of $\phi$. Since $ w^{\top} z \leq \|w\|\|z\|$ and $\|\int_{0}^{1} f(x)\,\mathrm{d}x \|\leq \int_{0}^{1} \|f(x)\|\,\mathrm{d}x$, by the definition of the set $\mathcal{Z}$, we have
   \begin{align*}
      \phi\left(f(x+y)-f(x) - \nabla f(x)^{\top}y\right) \leq \int_{0}^1 \|\nabla f(x+ty)-\nabla f(x)\|~\|y\|\,\mathrm{d}t \leq \frac{L}{2}\|y\|^{2}.  
   \end{align*}
   Hence, the result holds.
\end{proof}

Based on ordering relations \eqref{order_relation}, we define \emph{minimal set} for a given set. For any set \( U \subseteq \mathbb{R}^m \), the set of all non-dominated points in \( U \) with respect to the cone \( C \) is called the \emph{minimal set} of \( U \), denoted by \( \mathcal{M}(U, C) \), and is given by (see \cite{steepest2021set_optimization})
\begin{equation}\label{minimal_set}
    \mathcal{M}(U, C) := \left\{ \bar{y} \in U ~\middle|~ \nexists\, \hat{y} \in U,\, \hat{y} \neq \bar{y} \text{ such that } \hat{y} \preceq_C \bar{y} \right\}.
\end{equation}
Note that if $U$ is compact, then the set $U$ satisfies a domination rule with the minimal set of $U$ associated with the cone $C$ as defined below:
\begin{equation}\label{Domination_rule}
    U+C=\mathcal{M}(U,C)+C.
\end{equation}

We conclude this section by presenting the following function, which will be relevant for the later analysis.
Define $\xi : \mathcal{P}(\mathbb{R}^m) \rightarrow \mathbb{R}$ by
\begin{equation}\label{zeta}
 \xi(Z) := \inf_{z \in Z} \phi(z).   
\end{equation}
The function $\xi$ possesses a monotonicity property, meaning that for any pair of sets $U, V \subseteq \mathbb{R}^m$,

$$
U \preceq^{\ell} V \implies \xi(U) \leq \xi(V).
$$
\section{Optimality condition and descent direction}\label{sec.3}

We begin with the analysis of the optimality conditions for \eqref{sop} by considering the vector counterpart \eqref{VOP}. A point $\bar{y}$ is known as a local weakly efficient point of \eqref{VOP} if there exists a neighborhood $N_{\epsilon}(\bar{y})$ of $\bar{y}$ such that
\[\nexists~ \hat{y}\in N_{\epsilon}(\bar{y})\quad: \quad h(\hat{y})\prec_{C} h(\bar{y}).\]
Likewise, if $N_{\epsilon}(\bar{y})=\mathbb{R}^{n}$, the point $\bar{y}$ is known as a weakly efficient point of \eqref{VOP}. In the following proposition, a necessary optimality condition is obtained for the problem \eqref{VOP}.\\
\begin{proposition}\label{pro_sta_vector}
    Consider the problem \eqref{VOP}. If a point $\bar{y}$ is a local weakly efficient point, then $D_{v}h(\bar{y})\notin -\mathrm{int}(C)$ for all $v\in \mathbb{R}^{n}$. The converse is true if the objective function $h$ is convex with respect to $C$.
\end{proposition}
\begin{proof}
 On the contrary, let there exist a $\bar{v}\in \mathbb{R}^{n}$ such that $D_{\bar{v}}h(\bar{y})\in -\mathrm{int}(C)$. Therefore, for sufficiently small $\alpha > 0$, it follows that
 \[ \frac{1}{\alpha}\left(h(\bar{y}+\alpha \bar{v})-h(\bar{y})\right)\prec_{C} 0.\]
Consequently, $\bar{y}$ is not a local weakly efficient point.\\

We now assume that the objective function $h$ is convex with respect to $C$ and the point $\bar{y}$ is not a weakly efficient point of the problem \eqref{VOP}.
Accordingly, there exists $\hat{y}\in \mathbb{R}^{n}$ such that $$h(\hat{y})\prec_{C} h(\bar{y}).$$
As $h$ is convex with respect to $C$, we have
\[\frac{1}{\alpha}\left(h\left(\bar{y}+\alpha(\hat{y}-\bar{y}))-h(\bar{y}\right)\right)\overset{\eqref{convex_direction_deri}}{\preceq_{C}}h(\hat{y})-h(\bar{y})\prec_{C}0 \quad  \text{for any }\alpha \in (0,1).\]
This implies, $D_{\bar{v}}h(\bar{y})\in -\mathrm{int}(C)$, where $\bar{v}=\hat{y}-\bar{y}$. Hence, the result follows. 
\end{proof}
In view of Proposition \ref{pro_sta_vector}, we call $\bar{y}\in \mathbb{R}^{n}$ a stationary point of \eqref{VOP} if
\begin{equation}\label{efficient_stationary}
    D_{v}h(\bar{y})\notin -\mathrm{int}(C) \quad \text{for all }v\in \mathbb{R}^{n}.\\
\end{equation}

We describe solution framework for \eqref{sop}: a point $\bar{y}$ is called local weakly minimal solution of \eqref{sop} if there is a neighborhood $N_{\epsilon}(\bar{y})$ of $\bar{y}$ such that for no $y\in N_{\epsilon}(\bar{y})$, $H(y)\prec^{\ell} H(\bar{y})$. Likewise, if $N_{\epsilon}(\bar{y})=\mathbb{R}^{n}$, the point $\bar{y}$ is called as a weakly minimal solution of \eqref{sop}.

Before proceeding to the derivation of optimality conditions for the problem \eqref{sop}, we present an interesting observation. 
It is an established fact that for a convex \eqref{VOP} (with $f=0$), every locally efficient solution is also globally efficient. However, this property does not necessarily hold for an \eqref{sop} with $f^{j}=0$ and $g^{j}$ is convex for all $j\in [p]$. 
In particular, even if \( \bar{y} \) and \( \hat{y} \) are both local optimal solutions of a convex \eqref{sop} problem, it is possible that
$H(\hat{y}) \prec^\ell H(\bar{y})$. For instance, consider an illustration discussed below. \\

\noindent
\textbf{Example.}
  Let the problem \eqref{sop} with the objective map \(H:\mathbb{R} \rightrightarrows \mathbb{R}\) is defined as 
  \begin{equation}\label{example_convex}
    H(x) := \{ g^{1}(x),\, g^{2}(x) \},  
  \end{equation} 
where the ordering cone is \(C = \mathbb{R}_{+}\).  
The  functions \(g^{j} : \mathbb{R} \to \mathbb{R}\), for \(j = 1,2\), given by  
\[
g^{1}(x): = (x- 2.5)^{2} - 1 \text{ and } g^{2}(x) := (x- 1)^{2} - 2.
\]

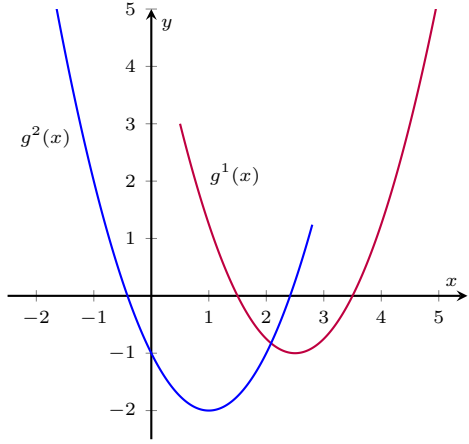
\begin{figure}  
\centering 
\begin{tikzpicture}
\begin{axis}[
color= black, 
thick, 
xmin=-2.5, 
xmax=5.5, 
ymin=-2.5, 
ymax=5, 
axis equal image, 
font=\footnotesize, 
xtick distance=1, 
ytick distance=1, 
inner axis line style={stealth-stealth}, 
xlabel = {$x$}, 
ylabel = {$y$}, 
axis x line=middle,
axis y line=middle 
] 
\addplot[
domain=0.5:5,   
samples=200, 
color=purple, 
]{(x - 2.5)^2 - 1};
\addplot[ 
domain=-1.9:2.8,  
samples=200, 
color=blue, 
]{(x - 1)^2 - 2}; 
\end{axis}
\node at (0.5,4) {\footnotesize $g^2(x)$};
\node at (3,3.5) {\footnotesize $g^1(x)$};
\end{tikzpicture} 
\caption{Graph of $g^{1}$ and $g^{2}$ in \(H\) as given in \eqref{example_convex}}\label{wrap-fig:1} 
\end{figure}
\noindent
Note that $g^{1}$ and $g^{2}$ are convex.
From Fig.~\ref{wrap-fig:1}, it is clear that \(\bar{y} = 1\) and \(\hat{y} = 2.5\) are both local weakly minimal solutions of~\eqref{sop}.
Moreover, $H(\bar{y})=\{-2,~1.25\}$ and  $H(\hat{y})=\{-1,~0.25\}$. 
Therefore, it follows that $H(\bar{y}) \prec^\ell H(\hat{y})$. Thus, in this case, every local weakly minimal solution is not a weakly minimal solution.

To derive optimality and stationarity conditions, we revisit the solution approach proposed by Bouza et al. \cite{steepest2021set_optimization}. We begin by introducing some related definitions. Thereafter, we demonstrate why and how they can be used to formulate optimality conditions for the problem \eqref{sop}.\\

The following index-based set-valued maps are defined with respect to the objective map \( H \) in the problem~\eqref{sop}.

\begin{definition} \textnormal{\cite{steepest2021set_optimization}}
\begin{enumerate}[(i)]
\item The map \( I : \mathbb{R}^n \rightrightarrows [p] \), called active index map, is defined by the set of all indices \( j \in [p] \) such that \( h^j(x) \in \mathcal{M}(H(x), C), ~x\in \mathbb{R}^{n}\). 
    
\item For a given \( u \in \mathbb{R}^m \), the map \( I_u(x): \mathbb{R}^n \rightrightarrows [p] \), referred as the value index map, is defined by the set of all indices \( j \in [p] \) such that \( h^j(x) = u \), $x\in \mathbb{R}^{n}$. \\ 
\end{enumerate}
\end{definition}

\begin{definition} \textnormal{\cite{steepest2021set_optimization}}
For a given point \( \bar{y}\), a cardinality function \( \omega: \mathbb{R}^n \rightarrow [p] \) is defined as the number of elements in the minimal set \( \mathcal{M}(H(\bar{y}), C) \), i.e.,
\begin{equation*}
 \omega(\bar{y}) := \left| \mathcal{M}(H(\bar{y}), C) \right|.   
\end{equation*}
\end{definition}


\begin{definition} \textnormal{\cite{steepest2021set_optimization}} 
Let \( x \in \mathbb{R}^n \) and \( \{ u^x_1, u^x_2, \ldots, u^x_{\omega(x)} \} \) be an enumeration of \( \mathcal{M}(H(x), C) \). The \emph{partition set} \( P_x \) is defined by
\begin{equation}\label{pat_set}
   P_x := \prod_{j = 1}^{\omega(x)} I_{u^x_j}(x), 
\end{equation}
where \( I_{u^x_j}(x) \subseteq [p] \) denotes the set of indices \( j \) such that \( h^j(x) = u^x_j \).\\ 
\end{definition}

We now discuss the significance of the \textit{minimal set} and \textit{partition set} at a given point $\bar{y}$ in obtaining a weakly minimal solution to the problem \eqref{sop}.

\begin{enumerate}[(i)]
    \item The \textit{minimal set} is essential in the process of identifying weakly minimal solutions to the problem~\eqref{sop}, owing to the following observation.\\

    \noindent
 For any \( \bar{y} \in \mathbb{R}^n \), the domination rule~\eqref{Domination_rule} implies that
    \(
    \mathcal{M}(H(\bar{y}), C) + C = H(\bar{y}) + C.
   \)
    Accordingly, for \(x \in \mathbb{R}^n \),
    \[
    \mathcal{M}(H(x), C) \prec^\ell \mathcal{M}(H(\bar{y}), C)\quad \iff \quad  H(x) \prec^\ell H(\bar{y}).
    \]

 In particular, if there is a neighborhood $N_{\epsilon}(\bar{y})\subseteq \mathbb{R}^n$ of $\bar{y}$ such that no point 
$x \in N_{\epsilon}(\bar{y}) $ satisfies
\[
\mathcal{M}(H(x), C) \prec^\ell \mathcal{M}(H(\bar{y}), C),
\]
then $\bar{y}$ is a local weakly minimal solution to \eqref{sop}.

Therefore, in verifying if a chosen point $\bar{y}$ is a weakly minimal solution or not, and if not, then in identifying a potentially better alternative $x\in \mathbb{R}^{n}$, it is sufficient to investigate the minimal set of \( H \) at $\bar{y}$.\\

\item  The \textit{partition set} plays a crucial role in handling the minimal set \( \mathcal{M}(H(x), C) \).
This role can be understood as follows. Let $\bar{y}$ be a chosen point. For each $a\in P_{\bar{y}}$, define 
\[
        \Tilde{h}^a : \mathbb{R}^n \rightarrow \underbrace{\mathbb{R}^m \times \mathbb{R}^m \times \cdots \times \mathbb{R}^m}_{\omega(\bar{y})~\text{times}} \quad \text{by }\Tilde{h}^a(x) := \left( h^{a_1}(x), h^{a_2}(x), \ldots, h^{a_{\omega(\bar{y})}}(x) \right)^{\top}.
    \]
At $a\in P_{\bar{y}}$, the optimization problem \eqref{VOP} takes the form \begin{equation}\tag{$\mathrm{VOP}^{a}_{\Tilde{C}}$}
        \label{vector_opt_family}
        \preceq_{\Tilde{C}} \text{ - } \min_{x \in \mathbb{R}^n} \Tilde{h}^a(x),
    \end{equation}
    where $ \Tilde{C} := \underbrace{C \times C \times \cdots \times C}_{\omega(\bar{y})~\text{times}}$.\\
    Note that for any partition element $a\in P_{\bar{y}}$, if there exists $\hat{y}\in \mathbb{R}^{n}$ such that $\Tilde{h}^{a}(\hat{y})\prec_{\Tilde{C}} \Tilde{h}^{a}(\bar{y})$, then $H(\hat{y})\prec^{\ell}H(\bar{y}).$
    This is because the image $h^{a_j}(\bar{y})$ corresponding to each index $a_j$ belongs to the partition element $a\in P_{\bar{y}}$, and is distinct by the definition of the partition set. Accordingly,
\begin{itemize}
        \item by using some element of the partition set, one can find a better alternative than $\bar{y}$ with the help of \eqref{vector_opt_family} if $\bar{y}$ is not a (local) weakly minimal solution of \eqref{sop}.\\
    \item  $\bar{y}$ is a local weakly minimal solution of \eqref{sop} only if it is a local weakly efficient solution of \eqref{vector_opt_family} for all $a\in P_{\bar{y}}$. Converse is also true as established below in Lemma \ref{Lemma_set_to_vec} under lower semi-continuity of $h^{j},~j\in [p]$.
\end{itemize}     
\end{enumerate}


\begin{lemma}\label{Lemma_set_to_vec}
Consider the problem \eqref{sop} and a point $\bar{y}$.
    Then, \( \bar{y} \) is a local weakly minimal solution of the problem \eqref{sop} if  \( \bar{y} \) is a local weakly efficient solution of the problem \eqref{vector_opt_family} for every \( a \in P_{\bar{y}} \).
\end{lemma}

\begin{proof}
Let \( \bar{y} \) be a local weakly efficient solution of the problem \eqref{vector_opt_family} for every \( a \in P_{\bar{y}} \). Assume, on the contrary, that \( \bar{y} \) is not a local weakly minimal solution to \eqref{sop}.
Accordingly, there exists a sequence $\{x^{k}\}$ converging to $\bar{y}$ satisfying 
\[
H\left(x^{k}\right)\prec^{\ell} H(\bar{y}) \quad \text{for all } k\in \mathbb{N}.
\]
Therefore, there exists a subsequence $\{x^{n_k}\}$ of $\{x^{k}\}$ and $i_{j}\in [p]$ such that 
\begin{equation}\label{eq_1_Lemma2.3}
    h^{i_{j}}(x^{n_k})\prec_{C}h^{a_j}(\bar{y}) \quad \text{for all } j\in [\omega(\bar{y})],
\end{equation}
where $a\in P_{\bar{y}}$, because $p$ is finite. Consequently, when $k\to +\infty$, for all $ j\in [\omega(\bar{y})]$ we deduce that
\begin{equation*}
   \underset{k\to +\infty}{\liminf}~h^{i_{j}} (x^{n_k})\preceq_{C}h^{a_j}(\bar{y}) .  
\end{equation*}
Since the function $h^{j},~j\in [p],$ is lower semi-continuous, we get 
\begin{equation*}
      h^{i_{j}} (\bar{y})\preceq_{C}\underset{k\to +\infty}{\liminf}~h^{i_{j}} (x^{n_k})\preceq_{C}h^{a_j}(\bar{y}) \implies h^{i_{j}}(\bar{y})=h^{a_j}(\bar{y}) \quad \text{for every }j\in [\omega(\bar{y})]
\end{equation*}
because $\{h^{a_j}(\bar{y})\}_{j\in [\omega(\bar{y})]}=\mathcal{M}(H(\bar{y}),K)$.\\
Accordingly, $a^{i_j}:=\{i_1, i_2, \ldots, i_{\omega(\bar{y})}\}\in P_{\bar{y}}$ associated with $i_j\in [p]$, $j=1,2,\ldots, \omega(\bar{y})$. Therefore, there exists a partition element $a^{i_j}\in P_{\bar{y}}$ such that
\[
\Tilde{h}^{a^{i_j}}(x^{n_k})\overset{\eqref{eq_1_Lemma2.3}}{\prec_{C}} \Tilde{h}^{a^{i_j}}(\bar{y}) ,
\]
which contradicts the assumption that \( \bar{y} \) is a local weakly efficient solution of the problem \eqref{vector_opt_family} for every \( a \in P_{\bar{y}} \). Hence, the proof is complete.
\end{proof}
\begin{remark}\label{N_S_c_optimal}
  Consider the problems \eqref{sop} and \eqref{vector_opt_family}. Then, $\bar{y}$ is a weakly minimal solution of \eqref{sop} if and only if $\bar{y}$ is a weakly efficient solution to \eqref{vector_opt_family} for all the partition elements $a\in P_{\bar{y}}$.   \\
\end{remark}
Next, we proceed to discuss the stationarity condition of the problem \eqref{sop} and highlight its significance in characterizing solutions. Furthermore, we present an alternative and more practical approach to verify stationarity in the context of the problem \eqref{sop}.\\

\begin{definition}\label{D_Stat}
    We say \(\bar{y}  \) to be stationary point of the problem \eqref{sop} if, for every \( a \in P_{\bar{y}} \), one can find an index \( j \in [\omega(\bar{y})] \) such that for all $v\in \mathbb{R}^{n}$, the directional derivative satisfies
\[
D_v h^{a_j}(\bar{y}) \notin -\mathrm{int}(C).
\]
\end{definition}

\begin{proposition}
    Consider the problem \eqref{sop}. Assume that $\bar{y} $ is a local weakly minimal solution of \eqref{sop}. Then, $\bar{y}$ must be stationary. Conversely, the same equivalence holds if $h^{j}$ is convex for every $j\in I(\bar{y})$.
\end{proposition}

\begin{proof}
    Assume that $\bar{y}$ is a local weakly minimal solution of the problem \eqref{sop}. Therefore, by Lemma \ref{Lemma_set_to_vec}, for all $a\in P_{\bar{y}}$, $\bar{y}$ is a local weakly efficient solution of \eqref{vector_opt_family}. From Proposition \ref{pro_sta_vector}, we accordingly obtain 
    $$\forall a \in P_{\bar{y}},~v\in \mathbb{R}^{n}: \quad D_{v}\Tilde{h}^{a}(\bar{y}) \notin -\mathrm{int}(\Tilde{K}),$$ 
    where $D_{v}\Tilde{h}^{a}(\bar{y}):=\left[D_{v}h^{a_1}(\bar{y}), D_{v}h^{a_2}(\bar{y}), \ldots, D_{v}h^{a_{\omega(\bar{y})}}(\bar{y})\right]$ and $\Tilde{K}:=K\times K \times \cdots \times K$ ($\omega(\bar{y})$ times).

    \noindent 
    Therefore, for all $ a\in P_{\bar{y}}$, there exists an index $j\in [\omega(\bar{y})]$ such that
    $$ D_{v}h^{a_j}\notin -\mathrm{int}(C) \quad \text{for all }v\in \mathbb{R}^{n}.$$ It follows that $\bar{y}$ is a stationary point to \eqref{sop}.  \\

    \noindent
Now, assume that the functions $h^{j}$ in \eqref{sop} are convex for all $j \in I(\bar{y})$, and suppose that $\bar{y}$ is not a local weakly minimal solution to \eqref{sop}.
Then, there exists an $\bar{a}\in P_{\bar{y}}$ such that the point $\bar{y}$ is not local weakly efficient solution to \eqref{vector_opt_family}. \\
    Note that $h^{j}$ is convex. Thus, from Proposition \ref{pro_sta_vector}, $\bar{y}$ is not a stationary point to \eqref{vector_opt_family}.
    Accordingly, there exists $\bar{v}\in \mathbb{R}^{n}$ such that $D_{\bar{v}}\Tilde{h}^{\bar{a}}(\bar{y})\in -\mathrm{int}(\Tilde{K})$, and hence there exists a partition element $\bar{a}\in P_{\bar{y}}$ such that
$$ D_{\bar{v}} h^{\bar{a}_j}(\bar{y}) \in -\mathrm{int}(C) \quad \text{for every } j\in [\omega(\bar{y})].$$ 

          \noindent
      Therefore, $\bar{y}$ is not a stationary point to \eqref{sop}.
\end{proof}



To capture stationary points of \eqref{sop}, we introduce some specific functions, which serve as a tool to facilitate the identification of stationary points.\\

For a given $l > 0$ and a point $\bar{y}\in \mathrm{dom}(H) $, consider a function 
\(\varphi_{l}(\cdot, \cdot; \bar{y}): P_{\bar{y}} \times \mathbb{R}^{n} \to \mathbb{R}\cup \{+\infty\}\) as
\begin{equation}\label{var_phi}
    \varphi_{l}(a, v; \bar{y}) := \max_{j \in [\omega(\bar{y})]}~ \phi\left( \nabla f^{a_j}(\bar{y})^{\top} v + g^{a_j}(\bar{y}+v) - g^{a_j}(\bar{y}) \right) + \frac{l}{2} \|v\|^2.
\end{equation}

\noindent
Note that for each $a \in P_{\bar{y}}$, the function $\varphi_{l}(v; a, \bar{y})$ is strongly convex in $v$. Consequently, $\varphi_{l}(v; a, \bar{y})$ attains a unique minimizer over $\mathbb{R}^{n}$. Moreover, since $\varphi_{l}(0; a, \bar{y}) = 0$, we get $\underset{v\in \mathbb{R}^{n}}{\min}\,\varphi_{l}(v; a, \bar{y}) \leq 0$ and
\[
 \min_{v \in \mathbb{R}^{n}} \varphi_{l}(v; a, \bar{y}) = 0 \iff v = 0
\]
\noindent
for any given point $\bar{y}$ and $a\in P_{\bar{y}}$. Note that the set $P_{\bar{y}}$ of partition elements is finite. Hence, the function $\varphi_{l}(a, v; \bar{y})$ also attains its minimum over the set $P_{\bar{y}} \times \mathbb{R}^{n}$. As a result, $\underset{(a, v) \in P_{\bar{y}} \times \mathbb{R}^{n}}{\min}\,\varphi_{l}(a,v; \bar{y}) \leq 0$ and
\begin{align*}
 \min_{(a, v) \in P_{\bar{y}} \times \mathbb{R}^{n}} \varphi_{l}(a,v; \bar{y}) = 0 \iff v = 0. 
\end{align*}
At any given point $\bar{y}$, the expression 
\( \underset{(a, v) \in P_{\bar{y}} \times \mathbb{R}^{n}}{\min} \varphi_{l}(a,v; \bar{y})\)
provides a unique numerical value associated with $\bar{y}$. This motivates us to define a function that maps each given point $\bar{y}\in \mathrm{dom}(H)$ to the corresponding minimum value. Define a function $\Theta_{l}:P_{\bar{y}}\times\mathbb{R}^{n} \to \mathbb{R}$ by
\begin{equation}\label{opt_val}
    \Theta_{l}(\bar{y}) := \min_{(a, v) \in P_{\bar{y}} \times \mathbb{R}^{n}}  \varphi_{l}(a,v; \bar{y}) .
\end{equation}
We denote
\begin{equation}\label{opt_sol}
    \left(a^{l}(\bar{y}), v_{l}(\bar{y})\right) \in \underset{(a, v) \in P_{\bar{y}} \times \mathbb{R}^{n}}{\arg\min}  \varphi_{l}(a, v; \bar{y}) .
\end{equation}
Accordingly, we have
\(\Theta_{l}(\bar{y}) = \varphi_{l}\left(a^{l}(\bar{y}), v_{l}(\bar{y}); \bar{y}\right)\),
and hence the following equivalences hold:
\begin{equation}\label{theta_e_0}
 v_{l}(\bar{y}) = 0   \iff  \Theta_{l}(\bar{y}) = 0 \quad \text{and}\quad  v_{l}(\bar{y}) \neq 0   \iff  \Theta_{l}(\bar{y}) < 0.
\end{equation}

For convenience, the following notational conventions will be adopted throughout the paper.

\begin{framed}
\noindent
   \begin{itemize}
       \item Given a reference point $\bar{y} \in \mathrm{dom}(H)$, we denote by $\bar{\omega}$ the cardinality $\omega(\bar{y})$ and by $\bar{P}$ the partition set $P_{\bar{y}}$. An optimal solution $(a^{l}(\bar{y}), v_{l}(\bar{y}))$ of the problem \eqref{opt_sol} is abbreviated as $(\bar{a}^{l}, \bar{v}_{l})$.\\
       \item Similarly, for an iterative point $x^{k} \in \mathrm{dom}(H)$, we abbreviate $\omega\left(x^{k}\right)$ and $P_{x^{k}}$ as $\omega_{k}$ and $P_{k}$, respectively. An optimal solution $\left(a^{l}\left(x^{k}\right), v_{l}\left(x^{k}\right)\right)$ of \eqref{opt_sol} is denoted by $\left(a^{l}_{k}, v^{k}_{l}\right)$.

   \end{itemize}
 
\end{framed}

In the following, we present some equivalent conditions for the stationarity of the problem~\eqref{sop}.

\begin{theorem}\label{equ_stationary}
   Let $\bar{y}\in \mathrm{dom}(H)$ be a given point. Consider the functions $\Theta_{l}$ and $v_{l}$ as defined in \eqref{opt_val} and \eqref{opt_sol}, respectively. Then, the following relations are equivalent:
    \begin{enumerate}[(i)]
        \item $\bar{y}$ is a stationary point of \eqref{sop}.
        \item $\bar{v}_{l}=0$ and $\Theta_{l}(\bar{y})=0.$
    \end{enumerate}   
\end{theorem}

  \begin{proof}
      $(i)\implies (ii)$ We prove by the method of contradiction. Suppose that $\bar{y}$ is a stationary point of the problem~\eqref{sop} and $\Theta_{l}(\bar{y}) \neq 0$.
 Therefore, $\Theta_{l}(\bar{y})< 0$ and $\bar{v}_{l}\neq 0$ by the relation \eqref{theta_e_0}. From the definition of $\Theta_{l}$, it follows that
      \[\underset{j \in [\bar{\omega}]}{\max}\, \phi\left( \nabla f^{a_j}(\bar{y})^{\top} \bar{v}_{l} + g^{a_j}(\bar{y}+\bar{v}_{l}) - g^{a_j}(\bar{y}) \right)<0.\]
      Note that $\alpha\mapsto \underset{j \in [\bar{\omega}]}{\max}\, \phi\left( \nabla f^{a_j}(\bar{y})^{\top} v + g^{a_j}(\bar{y}+v) - g^{a_j}(\bar{y}) \right)$ is a convex function and gets value zero at $v=0$. Therefore, for all $\alpha \in [0,1]$, we have
      \begin{align*}
         & \alpha\, \underset{j \in [\bar{\omega}]}{\max}\, \phi\left( \nabla f^{a_j}(\bar{y})^{\top} \bar{v}_{l} + g^{a_j}(\bar{y}+\bar{v}_{l}) - g^{a_j}(\bar{y}) \right) + (1-\alpha)0 \\
          \geq & ~\underset{j\in [\bar{\omega}]}{\max}\,\phi\left(\alpha\nabla f^{\bar{a}^{l}_j}(\bar{y})^{\top}\bar{v}_{l}+ g^{\bar{a}^{l}_j}\left(\bar{y}+\alpha \bar{v}_{l}\right)-g^{\bar{a}^{l}_j}(\bar{y})\right).
      \end{align*}
      Hence,
      \begin{align*}
          & \underset{j \in [\bar{\omega}]}{\max}\, \phi\left( \nabla f^{a_j}(\bar{y})^{\top} \bar{v}_{l} + g^{a_j}(\bar{y}+\bar{v}_{l}) - g^{a_j}(\bar{y}) \right) \\
          =&~\underset{\alpha\to 0_{+}}{\lim}\tfrac{1}{\alpha}\left(\underset{j \in [\bar{\omega}]}{\max}\, \phi\left( \nabla f^{a_j}(\bar{y})^{\top} \bar{v}_{l} + g^{a_j}(\bar{y}+\bar{v}_{l}) - g^{a_j}(\bar{y}) \right)\right)\\
          \geq &~\underset{j\in [\bar{\omega}]}{\max}\,\phi\left(D_{\bar{v}_{l}}h^{\bar{a}^{l}_j}(\bar{y})\right).
      \end{align*}
     Consequently, for all $j\in [\bar{\omega}]$, $$D_{\bar{v}_{l}}h^{\bar{a}^{l}_j}(\bar{y})\in -\mathrm{int}(C).$$ Thus, $\bar{y}$ must be a non-stationary point of the problem~\eqref{sop}, which contradicts our assumption.
\\

        \noindent
        $(ii)\implies (i)$ From the relation \eqref{theta_e_0}, it follows that the conditions $\Theta_{l}(\bar{y}) = 0$ and $v_{l}(\bar{y}) = 0$ are equivalent. Hence, we may equivalently impose either of these conditions. Assume the case when $\Theta_{l}(\bar{y}) = 0$. Then, by the definition of $\Theta_{l}$ in \eqref{opt_val}, we have
        \[\varphi_{l}(a,z;\bar{y})\geq \Theta_{l}(\bar{y})=0 \quad \text{for all } z\in \mathbb{R}^{n}.\]
     Then, for all $a\in P_{\bar{y}}$ and $z\in \mathbb{R}^{n}$, we get
        \begin{align*}
           &\underset{j\in [\bar{\omega}]}{\max}\,\phi\left(\alpha~\nabla f^{a_j}(\bar{y})^{\top}v+ g^{a_j}\left(\bar{y}+ \alpha\,v\right)-g^{a_j}(\bar{y})\right)+\frac{\alpha^{2} l}{2}\|v\|^{2} \geq 0\\
        \implies & \underset{j\in [\bar{\omega}]}{\max}\,\phi\left(\nabla f^{a_j}(\bar{y})^{\top}v+ \frac{1}{\alpha}\left\{g^{a_j}\left(\bar{y}+ \alpha\,v\right)-g^{a_j}(\bar{y})\right\}\right)+\frac{\alpha l}{2}\|v\|^{2} \geq 0\\ 
        \implies & \underset{j\in [\bar{\omega}]}{\max}\,\phi\left(D_{v}h^{a_j}(\bar{y})\right)\geq 0 \quad \text{as }\alpha \to 0.
        \end{align*}
        Accordingly, for all $a\in P_{\bar{y}},~j\in[\bar{\omega}]$ and $v\in \mathbb{R}^{n}$, we get
        \[D_{v}h^{a_j}(\bar{y})\notin -\mathrm{int}(C).\]
        Hence, $\bar{y}$ is a stationary point.  
  \end{proof}
  
  \begin{remark}\label{R_eqv}
      Let $\bar{y}\in \mathrm{dom}(H)$ be a given point. Then, there exists $(\bar{a}^{l},\bar{v}_{l})$ as given in \eqref{opt_sol}. Furthermore, if $\bar{y}$ is a non-stationary point of the problem \eqref{sop}, it follows that $\bar{v}_{l} \neq 0$ and $\Theta_{l}(\bar{y}) < 0$,
      \begin{equation}\label{R_eqv_I}
       \text{i.e.,} \quad \underset{j\in [\bar{\omega}]}{\max} \phi\left(\nabla f^{\bar{a}^{l}_{j}}(\bar{y})^{\top} \bar{v}_{l}+g^{\bar{a}^{l}_{j}}\left(\bar{y}+\bar{v}_{l}\right)-g^{\bar{a}^{l}_{j}}(\bar{y})\right)+  \frac{l}{2}\left\|\bar{v}_{l}\right\|^{2}<0.   
      \end{equation}
      \[ \]
  \end{remark}
  
Finally, this section is concluded by discussing the notion of a descent direction for the problem \eqref{sop}. Let $\bar{y}\in \mathrm{dom}(H)$ be a non-stationary point of \eqref{sop}. By Remark \ref{R_eqv}, $\Theta_{l}(\bar{y})<0$. Therefore,
there exists a $d \in \mathbb{R}^{n}$ such that for every $j \in [\bar{\omega}]$, we have
\[
\nabla f^{\bar{a}^{l}_{j}}(\bar{y})^{\top} d + g^{\bar{a}^{l}_{j}}(\bar{y} + d) - g^{\bar{a}^{l}_{j}}(\bar{y}) \in -\mathrm{int}(C).
\]
Therefore, for any $\alpha \in (0,1]$ and for all $j \in [\bar{\omega}]$, it follows that
\[
\alpha \nabla f^{\bar{a}^{l}_{j}}(\bar{y})^{\top} d + \alpha \left( g^{\bar{a}^{l}_{j}}(\bar{y} + d) - g^{\bar{a}^{l}_{j}}(\bar{y}) \right) \in -\mathrm{int}(C).
\]
Moreover, from relation \eqref{convex_direction_deri}, we obtain the following for all $j\in [\bar{\omega}]$:
\[
\alpha \nabla f^{\bar{a}^{l}_{j}}(\bar{y})^{\top} d + g^{\bar{a}^{l}_{j}}(\bar{y} + \alpha d) - g^{\bar{a}^{l}_{j}}(\bar{y}) \in -\mathrm{int}(C).
\]
Thus, for $\bar{\alpha} \in (0,1)$, sufficiently small, the following inequality holds for all $j \in [\bar{\omega}]$:
\begin{align*}
    & f^{\bar{a}^{l}_{j}}(\bar{y} + \bar{\alpha} d) + g^{\bar{a}^{l}_{j}}(\bar{y} + \bar{\alpha} d) \prec_{C} f^{\bar{a}^{l}_{j}}(\bar{y}) + g^{\bar{a}^{l}_{j}}(\bar{y}), \\
    \text{i.e.,} \quad & h^{\bar{a}^{l}_{j}}(\bar{y} + \bar{\alpha} d) \prec_{C} h^{\bar{a}^{l}_{j}}(\bar{y}).
\end{align*}
Consequently,
\[
H(\bar{y} + \bar{\alpha} d) \prec^{\ell} H(\bar{y}).
\]
This shows that, for every neighborhood of $\bar{y}$, there exists a point $\hat{x} = \bar{y} + \bar{\alpha} d$ satisfying $H(\hat{x}) \prec^{\ell} H(\bar{y})$.\\

Next, we formally introduce the notion of a $C$-descent direction for the problem \eqref{sop}.\\

\begin{definition}\label{C-desent_def}
Let $\bar{y}\in \mathrm{dom}(H)$ be a non-stationary point of \eqref{sop}. A vector $d \in \mathbb{R}^{n}$ is called a $C$-descent direction of \eqref{sop} at $\bar{y}$ if
\begin{equation}\label{c_descent_direction}
   \underset{j\in [\bar{\omega}]}{\max}\,\phi\left(\nabla f^{\bar{a}^{l}_{j}}(\bar{y})^{\top} d + g^{\bar{a}^{l}_{j}}(\bar{y} + d) - g^{\bar{a}^{l}_{j}}(\bar{y})\right)<0. 
\end{equation}
\end{definition}

\noindent
Observe that the relation \eqref{c_descent_direction} holds at $d=\bar{v}_{l}$. Therefore, it follows that $\bar{v}_{l}$ is a $C$-descent direction of \eqref{sop} at the point $\bar{y}\in \mathrm{dom}(H)$.

\section{Proximal gradient method}\label{sec.4}
In this section, we introduce proximal gradient methods, both with line search and without line search, for computing stationary points of the problem \eqref{sop}. For the variant without line search, we impose an additional assumption that the mapping $\nabla g^{j}$ is Lipschitz continuous for every $j \in [p]$.

\subsection{Proximal gradient method with line search}
In this method, we choose an $l >0$ for the subproblem \eqref{opt_val}. Upon solving \eqref{opt_val} at a given non-stationary point $\bar{y}\in \mathrm{dom}(H)$, a descent direction $\bar{v}_{l}$ and a corresponding partition element $\bar{a}^{l}$ are obtained. To identify a better candidate point in the sense of sufficient descent, a line search strategy is employed to determine an acceptable step-size along the direction $\bar{v}_{l}$ from $\bar{y}$. Various line search strategies, such as those based on the Armijo or Wolfe conditions (see~\cite{steepest2021set_optimization,Ghosh CGM}), can be employed for the problem~\eqref{sop}. In this work, for simplicity, we adopt a line search procedure based on the Armijo condition.

We now formalize the Armijo-type line search condition. Let $\rho \in (0,1)$ and $\alpha$ be a positive scalar. Then, a step-size $\alpha$ is said to satisfy the Armijo line search condition at $\bar{y}\in \mathrm{dom}(H)$ in the direction of $\bar{v}_{l}$ to \eqref{sop} if the following inequality holds:
\begin{equation}\label{Armijo_LS}
    h^{\bar{a}^{l}_j}(\bar{y} + \alpha \bar{v}_{l}) \preceq_{C} h^{\bar{a}^{l}_j}(\bar{y}) + \rho \,\alpha \left( \nabla f^{\bar{a}^{l}_j}(\bar{y})^{\top} \bar{v} + g^{\bar{a}^{l}_j}(\bar{y} + \bar{v}_{l}) - g^{\bar{a}^{l}_j}(\bar{y}) \right) \quad \text{for all } j \in [\bar{\omega}].
\end{equation}

Below, we show that if a point $\bar{y}\in \mathrm{dom}(H)$ is non-stationary of \eqref{sop}, then there exists a non-singleton interval $\bar{I}\subseteq(0,1]$ such that the Armijo line search condition~\eqref{Armijo_LS} at $\bar{y}$ in the direction $\bar{v}_{l}$ is satisfied for all $\alpha \in \bar{I}$.
Furthermore, under this condition, for every $\alpha \in \bar{I}$, the point $\bar{y} + \alpha \bar{v}_{l}$ serves as an improved iterate, in the sense that $H(\bar{y} + \alpha \bar{v}_{l})$ strictly dominates $H(\bar{y})$.\\

\begin{theorem}\label{LS_Theorem}
    Consider the problem \eqref{sop}. Assume that $\bar{y}\in \mathrm{dom}(H)$ is a non-stationary point and let $\rho \in (0,1)$. Then, there exists a non-singleton interval $\bar{I}\subseteq (0,1]$ such that for all $\alpha\in \bar{I}$, the following relations hold:
    \begin{enumerate}[(i)]
        \item $h^{\bar{a}^{l}_{j}}(\bar{y}+\alpha \bar{v}_{l})\preceq_{C} h^{\bar{a}^{l}_{j}}(\bar{y})+\rho \,\alpha\left(\nabla f^{\bar{a}^{l}_{j}}(\bar{y})^{\top} \bar{v}_{l}+g^{\bar{a}^{l}_j}(\bar{y}+\bar{v}_{l})-g^{\bar{a}^{l}_j}(\bar{y})\right)$ for all $j\in [\bar{\omega}]$.
        \item  $H(\bar{y}+\alpha \bar{v}_{l})\prec^{{\ell}}H(\bar{y})$.
    \end{enumerate}
\end{theorem}
\begin{proof}
  $(i)$ This result is established by the method of contradiction. Assume that there does not exist any non-singleton interval $\bar{I} \subseteq (0, 1]$ such that the condition stated in $(i)$ is satisfied.

\noindent
Since $\bar{\omega}$ is finite, there exists an index $\hat{j} \in [\bar{\omega}]$ and a sequence $\{\alpha_{k}\},~\alpha_{k}\in (0,1]$, with $\alpha_{k} \to 0$ as $k \to +\infty$, such that the following relation holds:
  \[h^{\bar{a}^{l}_{\hat{j}}}(\bar{y}+\alpha_{k} \bar{v}_{l})\npreceq_{C} h^{\bar{a}^{l}_{\hat{j}}}(\bar{y})+\rho \,\alpha_{k}\left(\nabla f^{\bar{a}^{l}_{\hat{j}}}(\bar{y})^{\top} \bar{v}_{l}+g^{\bar{a}^{l}_{\hat{j}}}(\bar{y}+\bar{v}_{l})-g^{\bar{a}^{l}_{\hat{j}}}(\bar{y})\right) .\]
 Since $\alpha_{k} > 0$, it follows from the definition of the cone $C$ that
\begin{align*}
  & \frac{1}{\alpha_k}\left\{f^{\bar{a}^{l}_{\hat{j}}}(\bar{y}+\alpha_{k} \bar{v}_{l})-f^{\bar{a}^{l}_{\hat{j}}}(\bar{y})\right\} + \frac{1}{\alpha_k}\left\{g^{\bar{a}^{l}_{\hat{j}}}(\bar{y}+\alpha_{k} \bar{v}_{l})-g^{\bar{a}^{l}_{\hat{j}}}(\bar{y})\right\}\\
    \npreceq_{C} &~ \rho \left(\nabla f^{\bar{a}^{l}_{\hat{j}}}(\bar{y})^{\top} \bar{v}_{l}+g^{\bar{a}^{l}_{\hat{j}}}(\bar{y}+\bar{v}_{l})-g^{\bar{a}^{l}_{\hat{j}}}(\bar{y})\right).
\end{align*}
Note that $g^{j}$ is convex for every $j \in [p]$. Therefore, by applying the relation \eqref{convex_direction_deri} to the left-hand side of the inequality, we deduce that
  \begin{align*}
      &\frac{1}{\alpha_k}\left\{f^{\bar{a}^{l}_{\hat{j}}}(\bar{y}+\alpha_{k} \bar{v}_{l})-f^{\bar{a}^{l}_{\hat{j}}}(\bar{y})\right\} + g^{\bar{a}^{l}_{\hat{j}}}(\bar{y}+\bar{v}_{l})-g^{\bar{a}^{l}_{\hat{j}}}(\bar{y})\\
      \npreceq_{C}&~ \rho \left(\nabla f^{\bar{a}^{l}_{\hat{j}}}(\bar{y})^{\top} \bar{v}_{l}+g^{\bar{a}^{l}_{\hat{j}}}(\bar{y}+\bar{v}_{l})-g^{\bar{a}^{l}_{\hat{j}}}(\bar{y})\right).
      \end{align*}
This implies that
 \[ \frac{1}{\alpha_k}\left\{f^{\bar{a}^{l}_{\hat{j}}}(\bar{y}+\alpha_{k} \bar{v}_{l})-f^{\bar{a}^{l}_{\hat{j}}}(\bar{y})\right\} + (1-\rho)\left\{g^{\bar{a}^{l}_{\hat{j}}}(\bar{y}+ \bar{v}_{l})-g^{\bar{a}^{l}_{\hat{j}}}(\bar{y})\right\}\npreceq_{C} \rho ~\nabla f^{\bar{a}^{l}_{\hat{j}}}(\bar{y})^{\top} \bar{v}_{l}.\]
 As $k \to +\infty$, we have $\alpha_{k} \to 0$, and consequently,
  \begin{align}\label{tho_4.1_1}
      & (1-\rho)\left(\nabla f^{\bar{a}^{l}_{\hat{j}}}(\bar{y})^{\top} \bar{v}_{l}+g^{\bar{a}^{l}_{\hat{j}}}(\bar{y}+\bar{v}_{l})-g^{\bar{a}^{l}_{\hat{j}}}(\bar{y})\right)\notin -\mathrm{int}(C) \nonumber\\
      \text{i.e.,}\quad & \nabla f^{\bar{a}^{l}_{\hat{j}}}(\bar{y})^{\top} \bar{v}_{l}+g^{\bar{a}^{l}_{\hat{j}}}(\bar{y}+\bar{v}_{l})-g^{\bar{a}^{l}_{\hat{j}}}(\bar{y})\notin -\mathrm{int}(C).
  \end{align}
On the other hand, observe that $\bar{y}$ is a non-stationary point of \eqref{sop}. Therefore, by Remark \ref{R_eqv}, we have $\bar{v}_{l} \neq 0$ and $\Theta_{l}(\bar{y}) < 0$. Consequently, it follows from \eqref{R_eqv_I} that
\begin{align*}
    & \max_{j \in [\bar{\omega}]}~ \phi\left( \nabla f^{\bar{a}^{l}_j}(\bar{y})^{\top} \bar{v}_{l} + g^{\bar{a}^{l}_j}\left(\bar{y}+\bar{v}_{l}\right) - g^{\bar{a}^{l}_j}(\bar{y}) \right)+ \frac{l}{2} \left\|\bar{v}_{l}\right\|^2<0\\
  \implies &  \phi\left( \nabla f^{\bar{a}^{l}_j}(\bar{y})^{\top} \bar{v}_{l} + g^{\bar{a}^{l}_j}\left(\bar{y}+\bar{v}_{l}\right) - g^{\bar{a}^{l}_j}(\bar{y}) \right) <0  \quad \text{for all } j\in [\bar{\omega}].
\end{align*}
Therefore, for all $j\in [\bar{\omega}]$, we obtain from Lemma \ref{scal_func}(iv) that
\begin{equation}\label{tho_4.1_2}
    \nabla f^{\bar{a}^{l}_{j}}(\bar{y})^{\top} \bar{v}_{l} + g^{\bar{a}^{l}_{j}}\left(\bar{y}+\bar{v}_{l}\right) - g^{\bar{a}^{l}_{j}}(\bar{y})\in -\mathrm{int}(C),
\end{equation}
which contradicts the relation \eqref{tho_4.1_1}. Hence, for all $j\in [\bar{\omega}]$, there exists a non-singleton interval $\bar{I}\subseteq(0,1]$ such that
\begin{equation}\label{tho_4.1_3}
    h^{\bar{a}^{l}_{j}}(\bar{y}+\alpha \bar{v}_{l})\preceq_{C} h^{\bar{a}^{l}_{j}}(\bar{y})+\rho \,\alpha\left(\nabla f^{\bar{a}^{l}_{j}}(\bar{y})^{\top} \bar{v}_{l}+g^{\bar{a}^{l}_j}(\bar{y}+\bar{v}_{l})-g^{\bar{a}^{l}_j}(\bar{y})\right) \quad \text{for all }\alpha\in \bar{I}.
\end{equation}

\noindent 
$(ii)$ From the relation \eqref{tho_4.1_2}, for all $j\in [\bar{\omega}]$, it follows that
\begin{align*}
  & \nabla f^{\bar{a}^{l}_{j}}(\bar{y})^{\top} \bar{v}_{l} + g^{\bar{a}^{l}_{j}}\left(\bar{y}+\bar{v}_{l}\right) - g^{\bar{a}^{l}_{j}}(\bar{y})\in -\mathrm{int}(C)\\
 \implies & \rho \,\alpha\left(  \nabla f^{\bar{a}^{l}_{j}}(\bar{y})^{\top} \bar{v}_{l} + g^{\bar{a}^{l}_{j}}\left(\bar{y}+\bar{v}_{l}\right) - g^{\bar{a}^{l}_{j}}(\bar{y})\right)\in -\mathrm{int}(C) \\
 \implies & h^{\bar{a}^{l}_j}(\bar{y})+\rho \,\alpha\left(  \nabla f^{\bar{a}^{l}_{j}}(\bar{y})^{\top} \bar{v}_{l} + g^{\bar{a}^{l}_{j}}\left(\bar{y}+\bar{v}_{l}\right) - g^{\bar{a}^{l}_{j}}(\bar{y})\right)\in h^{\bar{a}^{l}_j}(\bar{y})-\mathrm{int}(C) \\
 \implies & h^{\bar{a}^{l}_j}(\bar{y})+\rho \,\alpha\left(  \nabla f^{\bar{a}^{l}_{j}}(\bar{y})^{\top} \bar{v}_{l} + g^{\bar{a}^{l}_{j}}\left(\bar{y}+\bar{v}_{l}\right) - g^{\bar{a}^{l}_{j}}(\bar{y})\right) \prec_{C} h^{\bar{a}^{l}_j}(\bar{y}),
\end{align*}
where $\rho\in (0,1)$ and $\alpha \in \bar{I}\subseteq (0,1]$. For all $\alpha \in \bar{I}$, we
accordingly obtain
\begin{align*}
   & \left\{h^{\bar{a}^{l}_j}(\bar{y})+\rho \,\alpha\left(  \nabla f^{\bar{a}^{l}_{j}}(\bar{y})^{\top} \bar{v}_{l} + g^{\bar{a}^{l}_{j}}\left(\bar{y}+\bar{v}_{l}\right) - g^{\bar{a}^{l}_{j}}(\bar{y})\right)\right\}_{j\in [\bar{\omega}]} \prec^{\ell} \left\{h^{\bar{a}^{l}_j}(\bar{y})\right\}_{j\in [\bar{\omega}]}\\
   \implies & \left\{h^{\bar{a}^{l}_j}(\bar{y})\right\}_{j\in [\bar{\omega}]} \subseteq  \left\{h^{\bar{a}^{l}_j}(\bar{y})+\rho \,\alpha\left(  \nabla f^{\bar{a}^{l}_{j}}(\bar{y})^{\top}\bar{v}_{l}+ g^{\bar{a}^{l}_{j}}\left(\bar{y}+\bar{v}_{l}\right) - g^{\bar{a}^{l}_{j}}(\bar{y})\right)\right\}_{j\in [\bar{\omega}]} + \mathrm{int}(C)\\
   \overset{\eqref{tho_4.1_3}}{\implies} & \left\{h^{\bar{a}^{l}_j}(\bar{y})\right\}_{j\in [\bar{\omega}]}\subseteq   \left\{h^{\bar{a}^{l}_j}\left(\bar{y}+\alpha \bar{v}_{l}\right) \right\}_{j\in [\bar{\omega}]} + C+\mathrm{int}(C)\\
   \overset{\eqref{Domination_rule}}{\implies} & H(\bar{y}) \subseteq   \left\{h^{\bar{a}^{l}_j}\left(\bar{y}+\alpha\bar{v}_{l}\right) \right\}_{j\in [\bar{\omega}]} + \mathrm{int}(C) \subseteq H\left(\bar{y}+\alpha\bar{v}_{l}\right) + \mathrm{int}(C)
\end{align*}
because $C+\mathrm{int}(C)=\mathrm{int}(C)$. Hence, there exists a non-singleton interval $\bar{I}\subseteq(0,1]$ such that $H(\bar{y}+\alpha\bar{v}_{l})\prec^{\ell}H(\bar{y})$ for all $\alpha\in \bar{I}.$
\end{proof}

In the following, Algorithm~\ref{algo 1} outlines the proposed method, which is based on the proximal gradient framework augmented with a line search strategy.
  
\begin{algorithm}[h!] 
\caption{A proximal gradient method with line search for solving \eqref{sop}}\label{algo 1}
\textbf{Step 0}. Choose $x^{0}\in \mathrm{dom}(H),~\rho \in (0,1),~\mu\in (0,1),$ and $l>0$. Provide a tolerance $\epsilon>0$. Set $k := 0.$ \\

\textbf{Step 1}. Find $$M_k:=\mathcal{M}(H\left(x^{k}\right),C),~\omega_{k} := \lvert \mathcal{M}(H\left(x^{k}\right),C)\rvert,\text{ and }P_{k}:=P_{x^{k}}.$$

\textbf{Step 2}. Compute
$$(a^{l}_{k},v^{k}_{l})\in \underset{(a,v)\in P_{k}\times \mathbb{R}^{n}}{\arg\min}\left\{\underset{j\in [\omega_k]}{\max}\,\phi\left(\nabla f^{a_{j}}\left(x^{k}\right)^{\top}v+g^{a_{j}}(x^{k}+v)- g^{a_{j}}\left(x^{k}\right)\right)+\frac{l}{2}\|v\|^2\right\}.$$

\textbf{Step 3}. If $\|v_{l}^{k}\| < \epsilon$ or $\lvert\Theta_{l}\left(x^{k}\right)\rvert<\epsilon$, then stop.\\

\textbf{Step 4}. Find
$$\alpha_{k}:=\underset{i\in \mathbb{N}\cup\{0\}}{\min}\left\{\mu^{i}~\middle|~\begin{aligned}&h^{a_{k,j}}(x^{k}+\mu^{i} v_{l}^{k})\preceq_{C} h^{a_{k,j}}\left(x^{k}\right)\\&+\rho \mu^{i}\left(\nabla f^{a_{k,j}}\left(x^{k}\right)^{\top}v_{l}^{k}+g^{a_{k,j}}\left(x^{k}+v_{l}^{k}\right)- g^{a_{k,j}}\left(x^{k}\right)\right)\quad \forall j\in [\omega_{k}]\end{aligned}\right\}.$$

\textbf{Step 5}. Set $x^{k+1} := x^{k}+\alpha_{k}v_{l}^{k},~ k:= k+1,$ and go to Step 1.
\end{algorithm}

The well-definedness of Algorithm \ref{algo 1} depends on Steps 2 and 4. Since $x^{0}\in \mathrm{dom}(H)$, Remark~\ref{R_eqv} guarantees the existence of a $(a^{l}_{k}, v^{k}_{l})$ in Step 2. Moreover, whenever $\|v^{k}_{l}\| \neq 0$, Theorem~\ref{equ_stationary} ensures that the iterate $x^{k}$ is a non-stationary point of \eqref{sop}. Consequently, Step 3 combined with Theorem~\ref{LS_Theorem} guarantees the existence of a step-size $\alpha_{k}$ satisfying the condition specified in Step 4. Therefore, Algorithm~\ref{algo 1} is well-defined.

\subsection{Proximal gradient method without line search}
In this section, we propose a proximal gradient algorithm (Algorithm \ref{algo}) that omits line searches by fixing the step‐size to unity: given an iterate \(x^{k}\in \mathrm{dom}(H)\) and its descent direction \(v^{k}_{l}\), we set \(x^{k+1}:= x^{k} + v^{k}_{l},
\) and employ the strict descent condition:
\begin{equation}\label{m_d_p}
    H\left(x^{k+1}\right) \prec^{\ell} H\left(x^{k}\right).
\end{equation}
In order to guarantee \eqref{m_d_p} at every iteration, we suitably adjust the parameter \(l\) in \eqref{var_phi}. For the existence of such an \(l\), we assume that \(\nabla f^{j}\) is Lipschitz continuous with constant \(L^{j}\) for all $j\in [p]$, and define \(L := \max\{L^{j} \mid j \in [p]\}\).

In the sequel, we derive sufficient conditions on \(l\) in \(\eqref{var_phi}\) under which the descent property \(\eqref{m_d_p}\) is preserved. We start with recalling the three-point property that plays an important role in proximal-type methods.\\
\begin{tpp}
    Consider a proper convex function $\psi: \mathbb{R}^{n}\to \mathbb{R}\cup \{+\infty\}$ and define 
    \[z^{*}:=\underset{z\in \mathbb{R}^{n}}{\emph{\text{argmin}}}\left\{\psi(z)+\frac{1}{2}\|y-z\|^{2}\right\} \quad \text{for any given }y\in \mathbb{R}^{n}.\]
    Then, 
    \begin{equation}\label{tpp}
       2( \psi(z^{*})-\psi(z)) + \|z^{*}-z\|^{2} + \|z^{*}-y\|^{2}\leq \|y-z\|^{2} \quad \text{for all } z\in \mathbb{R}^{n}.
    \end{equation}
\end{tpp}

Observe that $$\alpha \mapsto \underset{j\in [\omega_k]}{\max}\,\phi\left(\nabla f^{a_{k,j}}\left(x^{k}\right)^{\top}v+g^{a_{k,j}}(x^{k}+v)- g^{a_{k,j}}\left(x^{k}\right)\right)=:\mathcal{F}^{k}(v)$$ is a proper convex function, and the expression for $v^{k}_{l}$ given in \eqref{opt_sol} can be equivalently represented as
\[
v^{k}_{l} = \underset{v \in \mathbb{R}^{n}}{\arg\min} \left\{ \frac{1}{l}\mathcal{F}^{k}(v) + \frac{1}{2} \|v\|^2 \right\}.
\]
Therefore, by setting $y = z = 0$ in \eqref{tpp}, we obtain
\[\frac{2}{l}(\mathcal{F}^{k}(v^{k}_{l})-\mathcal{F}^{k}(0))+ \|v^{k}_{l}-0\|^{2} + \|v^{k}_{l}-0\|^{2}\leq \|0-0\|^{2} \implies \mathcal{F}^{k}(v^{k}_{l})+l\|v^{k}_{l}\|^{2}\leq 0. \]
Thus, for all $k \in \mathbb{N}\cup\{0\}$, it follows that
\begin{equation}\label{tpp_c}
    \mathcal{F}^{k}(v^{k}_{l}) \leq- l \|v^{k}_{l}\|^{2}.
\end{equation}

Finally, the following result confirms that, for $l \geq \tfrac{L^{k}}{2}$, the iterates $x^{k+1}:=x^{k}+v^{k}_{l}$ satisfy the monotonicity condition \eqref{m_d_p}.\\

\begin{proposition}\label{w_line_mono}
Consider the problem \eqref{sop}. Let $x^{0}\in \mathrm{dom}(H)$ and $x^{k}:=x^{k-1}+v^{k-1}_{l}$ give a sequence of non-stationary points. Suppose that $\nabla f^{j}$ is Lipschitz continuous for every $j \in [p]$, with Lipschitz constant $L^{j}$ and define $L := \max\{L^{j} \mid j \in [p]\}$. Then, for every $k \in \mathbb{N}\cup\{0\}$, it follows that $H\left(x^{k+1}\right) \prec^{\ell} H\left(x^{k}\right)$.
\end{proposition}
\begin{proof}
    In view of the partition element $a_{k}\in P_{k}$, for all $j\in [\omega_{k}]$, we have
    \begin{align*}
     &\phi\left(h^{a^{l}_{k,j}}\left(x^{k+1}\right)-h^{a^{l}_{k,j}}\left(x^{k}\right)\right)\\ 
     =&~ \phi\left(f^{a^{l}_{k,j}}\left(x^{k}+v^{k}_{l}\right)-f^{a^{l}_{k,j}}\left(x^{k}\right)+\nabla f^{a^{l}_{k,j}}\left(x^{k}\right)^{\top}v^{k}_{l} \right. \\ 
     &\left.~-\nabla f^{a^{l}_{k,j}}\left(x^{k}\right)^{\top}v^{k}_{l}+g^{a^{l}_{k,j}}\left(x^{k}+v^{k}_{l}\right)-g^{a^{l}_{k,j}}\left(x^{k}\right)\right)\\
    \leq&~\phi\left(f^{a^{l}_{k,j}}\left(x^{k}+v^{k}_{l}\right)-f^{a^{l}_{k,j}}\left(x^{k}\right)-\nabla f^{a^{l}_{k,j}}\left(x^{k}\right)^{\top}v^{k}_{l}\right) \\ &~ +\phi\left(\nabla f^{a^{l}_{k,j}}\left(x^{k}\right)^{\top}v^{k}_{l}+g^{j}\left(x^{k}+v^{k}_{l}\right)-g^{j}\left(x^{k}\right)\right)\\ 
     \leq&~\frac{L}{2}\|v^{k}_{l}\|^{2} +\phi\left(g^{j}\left(x^{k}+v^{k}_{l}\right)-g^{j}\left(x^{k}\right)+\nabla f^{a^{l}_{k,j}}\left(x^{k}\right)^{\top}v^{k}_{l}\right) 
    \end{align*}
    from Lemma \ref{scal_func}(i) and Lemma \ref{descent lemma}. Therefore, we obtain
   \[\underset{j\in [\omega_k]}{\max} \phi\left(h^{a^{l}_{k,j}}\left(x^{k+1}\right)-h^{a^{l}_{k,j}}\left(x^{k}\right)\right)  \leq \frac{L}{2}\|v^{k}_{l}\|^{2} + \underset{j\in [\omega_k]}{\max}\phi\left(g^{j}\left(x^{k}+v^{k}_{l}\right)-g^{j}\left(x^{k}\right)+\nabla f^{a^{l}_{k,j}}\left(x^{k}\right)^{\top}v^{k}_{l}\right).\]
    By using the notion of $\mathcal{F}^{k}$, we get
    \begin{equation}\label{Prp4.1_eq1}
        \underset{j\in [\omega_k]}{\max} \phi\left(h^{a^{l}_{k,j}}\left(x^{k+1}\right)-h^{a^{l}_{k,j}}\left(x^{k}\right)\right)  \leq \frac{L}{2}\|v^{k}_{l}\|^{2} + \mathcal{F}^{k}(v^{k}_{l})\overset{\eqref{tpp_c}}{\leq}  \frac{L-2l}{2}\|v^{k}_{l}\|^{2}.
    \end{equation}
    Note that $x^{k}\in \mathrm{dom}(H)$ and is a non-stationary point. Therefore, by Theorem \ref{equ_stationary}, it follows that $v^{k}_{l} \neq 0$. Consequently, if $L < 2l$, we have
\[
\underset{j \in [\omega_k]}{\max} \, \phi\left( h^{a^{l}_{k,j}}\left(x^{k+1}\right) - h^{a^{l}_{k,j}}\left(x^{k}\right) \right) < 0.
\]
This implies that, for all $j \in [\omega_k]$,
\[
\phi\left( h^{a^{l}_{k,j}}\left(x^{k+1}\right) - h^{a^{l}_{k,j}}\left(x^{k}\right) \right) < 0 \quad \implies \quad h^{a^{l}_{k,j}}\left(x^{k+1}\right) \prec_{C} h^{a^{l}_{k,j}}\left(x^{k}\right).
\]
because of Lemma \ref{scal_func}(iv). Therefore, $H\left(x^{k+1}\right) \prec^{\ell} H\left(x^{k}\right)$.
\end{proof}
In the following (Algorithm~\ref{algo}), the proposed method is presented, which is developed within the proximal gradient framework and operates without incorporating any line search strategy.

\begin{algorithm}[h!] 
\caption{A proximal gradient method without line search for solving \eqref{sop}}\label{algo}
\textbf{Step 0}. Choose $x^{0}\in \mathrm{dom}(H)\text{ and }l>\tfrac{L}{2}$. Provide a tolerance $\epsilon>0$. Set $k := 0.$ \\

\textbf{Step 1}. Find $$M_k:=\mathcal{M}(H\left(x^{k}\right),C),~\omega_{k} := \lvert \mathcal{M}(H\left(x^{k}\right),C)\rvert,\text{ and }P_{k}:=P_{x^{k}}.$$

\textbf{Step 2}. Compute
$$(a^{l}_{k},v^{k}_{l})\in  \underset{(a,v)\in P_{k}\times \mathbb{R}^{n}}{\arg\min}\left\{\underset{j\in [\omega_k]}{\max}\,\phi\left(\nabla f^{a_{j}}\left(x^{k}\right)^{\top}v+g^{a_{j}}(x^{k}+v)- g^{a_{j}}\left(x^{k}\right)\right)+\frac{l}{2}\|v\|^2\right\}.$$

\textbf{Step 3}. If $\|v^{k}_{l}\| < \epsilon$ or $\lvert\Theta_{l}\left(x^{k}\right)\rvert<\epsilon$, then stop.\\

\textbf{Step 4}. Set $x^{k+1} := x^{k}+v^{k}_{l},~ k:= k+1,$ and go to Step 1.
\end{algorithm}

The well-definedness of Algorithm~\ref{algo} relies only on the execution of Step 2. Remark~\ref{R_eqv} guarantees the existence of an optimal solution $(a^{l}_{k}, v^{k}_{l})$ to the subproblem \eqref{opt_sol}. Consequently, Algorithm~\ref{algo} is well-defined.

\section{Convergence and complexity analysis}\label{sec.5}
In this section, we examine the convergence of the sequence of non-stationary points generated by both proposed algorithms---Algorithms \ref{algo 1} and \ref{algo}---for the problem~\eqref{sop}. We then analyze their computational complexity and discuss the corresponding convergence rates. The analysis relies on equivalent conditions of the stationarity presented in Theorem~\ref{equ_stationary}. 

We begin by establishing a relation involving the function $\xi$ defined in~\eqref{zeta}, which will be used frequently in the subsequent analysis.  
Let $\bar{y}\in \mathbb{R}^{n}$ be any given point. Since $\mathcal{M}(H(\bar{y}), C) \preceq^\ell H(\bar{y})$ and $\mathcal{M}(H(\bar{y}), C) \subseteq H(\bar{y})$, it follows from the monotonicity property of $\xi$ that
\begin{equation}\label{Mo_zeta}
    \xi \circ H(\bar{y}) 
    = \xi\left(\mathcal{M}(H(\bar{y}), C)\right) 
    = \min_{j\in [\bar{\omega}]} \phi\!\left(h^{\bar{a}^{l}_{j}}(\bar{y})\right).
\end{equation}

To establish the convergence of the proposed algorithms, we take the following commonly used assumption.\\
\begin{assumption}\label{Ass_label_bdd}
Define $\mathcal{L} := \{ x \in \mathbb{R}^{n} \mid H(x) \preceq^{\ell} H\left(x^{0}\right) \}$ for a given initial point from $ \mathbb{R}^{n}$.  
Then, for any sequence $\{S_{k}\} \subset H(\mathcal{L})$ satisfying $S_{k+1} \preceq^{\ell} S_{k}$ for all $k = 0, 1, 2, \ldots$, there exists a bounded set $S \subseteq \mathbb{R}^{m}$ such that 
\begin{equation}\label{lower_bdd_H}
   S \preceq^{\ell} S_{k} \quad \text{for all } k \in \mathbb{N}\cup\{0\}.
\end{equation}
\end{assumption}
We now establish global convergence of Algorithm~\ref{algo 1}.\\

\begin{theorem}\label{Con_algo1}
   Consider the problem \eqref{sop} and assume that the map $H$ satisfies \emph{Assumption~\ref{Ass_label_bdd}}. Let $\{x^{k}\}$ be a sequence of non-stationary points generated by \emph{Algorithm \ref{algo 1}}. Then, for any given precision $\epsilon>0$, there exists $k\in \mathbb{N}\cup \{0\}$ such that $\left\|v^{k}_{l}\right\| \leq \varepsilon$.
\end{theorem}

\begin{proof}
Suppose, to the contrary, that there exists a constant $\varepsilon > 0$ such that $\left\|v^{k}_{l}\right\| \geq \varepsilon$ for all $k \in \mathbb{N}\cup\{0\}$.
 Then, from Remark~\ref{R_eqv}, the following inequality holds:
    \begin{equation}\label{the5_e1}
       -\underset{j\in [\omega_{k}]}{\max} \phi\left(\nabla f^{a^{l}_{k,j}}\left(x^{k}\right)^{\top} v^{k}_{l}+g^{a^{l}_{k,j}}\left(x^{k}+v^{k}_{l}\right)-g^{a^{l}_{k,j}}\left(x^{k}\right)\right)\geq  \frac{l}{2}\left\|v^{k}_{l}\right\|^{2}\geq\frac{l}{2}\varepsilon^{2}>0
      .
    \end{equation}
    From Theorem \ref{LS_Theorem} and  Lemma \ref{scal_func}(ii), we have 
    \begin{align*}
       & \underset{j\in [\omega_{k}]}{\min}\phi\left(h^{a^{l}_{k,j}}\left(x^{k+1}\right)\right)\\
       \leq& ~ \underset{j\in [\omega_{k}]}{\min} \phi \left(h^{a^{l}_{k,j}}\left(x^{k}\right)+\rho\, \alpha_{k}\left(\nabla f^{a^{l}_{k,j}}\left(x^{k}\right)^{\top} v^{k}_{l}+g^{a^{l}_{k,j}}\left(x^{k}+v^{k}_{l}\right)-g^{a^{l}_{k,j}}\left(x^{k}\right)\right)\right)\\
        \leq & ~ \underset{j\in [\omega_{k}]}{\min} \phi\left( h^{a^{l}_{k,j}}\left(x^{k}\right)\right)+\rho\, \alpha_{k} \underset{j\in [\omega_{k}]}{\max} \phi\left(\nabla f^{a^{l}_{k,j}}\left(x^{k}\right)^{\top} v^{k}_{l}+g^{a^{l}_{k,j}}\left(x^{k}+v^{k}_{l}\right)-g^{a^{l}_{k,j}}\left(x^{k}\right)\right)
    \end{align*}
    because $\rho$ and $\alpha_{k}$ are positive. Therefore, applying the function $\xi$ defined in~\eqref{zeta} and the equation \eqref{Mo_zeta}, it follows that 
    \begin{align*}
              & \xi \circ H\left(x^{k+1}\right)  \leq \xi \circ H\left(x^{k}\right)+\rho\, \alpha_{k} \underset{j\in [\omega_{k}]}{\max} \phi\left(\nabla f^{a^{l}_{k,j}}\left(x^{k}\right)^{\top} v^{k}_{l}+g^{a^{l}_{k,j}}\left(x^{k}+v^{k}_{l}\right)-g^{a^{l}_{k,j}}\left(x^{k}\right)\right)\\
         \implies & -\rho\, \alpha_{k} \underset{j\in [\omega_{k}]}{\max} \phi\left(\nabla f^{a^{l}_{k,j}}\left(x^{k}\right)^{\top} v^{k}_{l}+g^{a^{l}_{k,j}}\left(x^{k}+v^{k}_{l}\right)-g^{a^{l}_{k,j}}\left(x^{k}\right)\right)\leq  \xi \circ H\left(x^{k}\right)  - \xi \circ H\left(x^{k+1}\right).
    \end{align*}
   Adding both sides of the above inequality from $k = 0$ to $\hat{k}-1,~\hat{k} \in \mathbb{N}$, we obtain
    \begin{align*}
         & -\rho~ \sum_{k=0}^{\hat{k}-1} \alpha_{k} \underset{j\in [\omega_{k}]}{\max} \phi\left(\nabla f^{a^{l}_{k,j}}\left(x^{k}\right)^{\top} v^{k}_{l}+g^{a^{l}_{k,j}}\left(x^{k}+v^{k}_{l}\right)-g^{a^{l}_{k,j}}\left(x^{k}\right)\right)\\
         \leq& ~  \sum_{k=0}^{\hat{k}-1}\left\{ \xi \circ H\left(x^{k}\right)  - \xi \circ H\left(x^{k+1}\right)\right\}=\xi \circ H\left(x^{0}\right)  - \xi \circ H\left(x^{\hat{k}}\right)\overset{\eqref{lower_bdd_H}}{\leq} \xi \circ H\left(x^{0}\right)  - \xi \circ S.    
    \end{align*}
    Therefore, from \eqref{the5_e1}, we get
    \begin{equation*}
        0<\frac{l}{2}\varepsilon^{2} \sum_{k=0}^{\hat{k}} \alpha_{k}\leq  \frac{\xi \circ H\left(x^{0}\right)  - \xi \circ S}{\rho}.
    \end{equation*}
    Consequently, $\underset{k\to +\infty}{\lim}\alpha_{k}=0.$
    Then, for any given arbitrary $\hat{i}\in \mathbb{N}$, $\mu^{\hat{i}}$ does not satisfy the condition given in Step 4 of Algorithm \ref{algo 1} for a sufficiently large $k$. Since $p$ is finite, there exists a sub-sequence $\{x^{n_k}\}$ of $\{x^{k}\}$ such that $a^{l}_{n_k}$ is fixed. Denote $\breve{a}^{l}:=a^{l}_{n_k}$ and $\breve{w}:= \omega_{n_k}$. Therefore, there exists $N_{\hat{i}}\in \mathbb{N}$ such that for all $k\geq N_{\hat{i}}$, the following relation holds:
    \begin{align*}
    & \underset{j\in [\breve{\omega}]}{\max}\Big\{h^{\breve{a}^{l}_{j}}\left(x^{n_k}+\mu^{\hat{i}} v^{l}_{n_k}\right)- h^{\breve{a}^{l}_{j}}(x^{n_k})\\
    &\hspace{2cm}-\rho \,\mu^{\hat{i}}\left(\nabla f^{\breve{a}^{l}_{j}}(x^{n_k})^{\top} v^{l}_{n_k}+g^{\breve{a}^{l}_{j}}\left(x^{n_k}+v^{l}_{n_k}\right)-g^{\breve{a}^{l}_{j}}(x^{n_k})\right)\Big\}\npreceq_C 0\\
     \overset{\eqref{convex_direction_deri}}{\implies} & \underset{j\in [\breve{\omega}]}{\max}\Big\{\tfrac{1}{\mu^{\hat{i}}}\left(f^{\breve{a}^{l}_{j}}\left(x^{n_k}+\mu^{\hat{i}} v^{l}_{n_k}\right)-f^{\breve{a}^{l}_{j}}\left(x^{n_k}\right)\right) \\&\hspace{2cm}+(1-\rho)\left\{g^{\breve{a}^{l}_{j}}\left(x^{n_k}+ v^{l}_{n_k}\right)-g^{\breve{a}^{l}_{j}}\left(x^{n_k}\right)\right\}-\rho \nabla f^{\breve{a}^{l}_{j}}(x^{n_k})^{\top} v^{l}_{n_k}\Big\}
     \npreceq_{C} 0 
    \end{align*}
   because $\mu\in (0,1)$. Accordingly, for sufficiently large $\hat{i}$, it follows that
  \[ \underset{j\in [\breve{\omega}]}{\max}\,\phi\left(\nabla f^{\breve{a}^{l}_{j}}(x^{n_k})^{\top}v^{l}_{n_k}+  g^{\breve{a}^{l}_{j}}\left(x^{n_k}+ v^{l}_{n_k}\right)-g^{\breve{a}^{l}_{j}}\left(x^{n_k}\right)\right)\geq 0,\] 
   which is
    a contradiction of \eqref{the5_e1}, thereby completing the proof.
\end{proof}
\begin{remark}\label{con_steep_dest}
If $g^{j} = 0$ for all $j \in [p]$, the set optimization problem~\eqref{sop} coincides with the one studied in \cite{steepest2021set_optimization}. In this case, the steepest descent method proposed in \cite{steepest2021set_optimization} is equivalent to Algorithm~\ref{algo 1}. It is worth noting that Theorem~\ref{Con_algo1} establishes the convergence of Algorithm~\ref{algo 1} without imposing any regularity assumption on the stationary point. Consequently, Theorem~\ref{Con_algo1} guarantees that the sequence of non-stationary points generated by the steepest descent method converges to a stationary point even in the absence of regularity.\\
\end{remark}

In the following, we have discussed the convergence of Algorithm \ref{algo}.\\

\begin{theorem}\label{con_alg2}
    Assume that the map $H$ satisfies \emph{Assumption~\ref{Ass_label_bdd}}. Suppose that $\nabla f^{j}$ is Lipschitz continuous for every $j \in [p]$, with Lipschitz constant $L^{j}$ and define $L := \max\{L^{j} \mid j \in [p]\}$. Let $\{x^{k}\}$ be a sequence of non-stationary points of the problem \eqref{sop} generated by \emph{Algorithm \ref{algo}}. Then, $\underset{k\to +\infty}{\lim}\left\|v^{k}_{l}\right\|=0$.
\end{theorem}

\begin{proof}
    In view of the $k^{\text{th}}$ iteration of Algorithm \ref{algo}, the relation \eqref{Prp4.1_eq1} implies
    \[ \underset{j\in [\omega_k]}{\max} \phi\left(h^{a^{l}_{k,j}}\left(x^{k+1}\right)-h^{a^{l}_{k,j}}\left(x^{k}\right)\right)  \leq   \frac{L-2l}{2}\left\|v^{k}_{l}\right\|^{2}.\]
    From Lemma \ref{scal_func} (i), it follows that 
 \[ \underset{j\in [\omega_k]}{\max} \phi \left(h^{a^{l}_{k,j}}\left(x^{k+1}\right)\right)-\underset{j\in [\omega_k]}{\min} \phi\left(h^{a^{l}_{k,j}}\left(x^{k}\right)\right)  \leq   \frac{L-2l}{2}\left\|v^{k}_{l}\right\|^{2}.\]
Therefore, applying the function $\xi$ defined in~\eqref{zeta} and the equation \eqref{Mo_zeta}, it follows that
 \begin{align*}
   &  \frac{L-2l}{2}\left\|v^{k}_{l}\right\|^{2} \geq \underset{j\in [\omega_k]}{\max} \phi \left(h^{a^{l}_{k,j}}\left(x^{k+1}\right)\right)-\xi \circ H\left(x^{k}\right)\geq \xi \circ H\left(x^{k+1}\right)- \xi \circ H\left(x^{k}\right) \\
   \implies  & \frac{L-2l}{2}  \sum_{k=0}^{\hat{k}-1}\left\|v^{k}_{l}\right\|^{2} \geq \sum_{k=0}^{\hat{k}-1}\left\{\xi \circ H\left(x^{k+1}\right)-\xi \circ H\left(x^{k}\right)\right\}= \xi \circ H\left(x^{\hat{k}}\right)-\xi \circ H\left(x^{0}\right), \quad \hat{k}\in \mathbb{N} \\
     \implies  & 0 >   \frac{L-2l}{2} \sum_{k=0}^{\hat{k}-1}\left\|v^{k}_{l}\right\|^{2} \overset{\eqref{lower_bdd_H}}{\geq} \xi \circ S-\xi \circ H\left(x^{0}\right) 
 \end{align*}
because $l>\tfrac{L}{2}$ and $x^{k}\in \mathrm{dom}(H)$ and is a non-stationary point. Accordingly, we obtain  
$\underset{k\to +\infty}{\lim} \left\|v^{k}_{l}\right\| = 0$,  
which completes the proof.
\end{proof}

Next, we address the computational complexity of the proposed techniques. For Algorithm~\ref{algo 1}, this analysis begins by establishing a lower bound on the step-size generated by the algorithm,  
as this bound plays a crucial role in determining its complexity. We determine it by imposing an additional condition on $H$ that $\nabla f^{j}$ is Lipschitz continuous for all $j \in [p]$ in the following lemma.\\

\begin{lemma}\label{lemma_step_size}
    Consider the problem \eqref{sop}. Assume that $\nabla f^{j}$ is Lipschitz continuous with Lipschitz constant $L^{j}$ for every $j \in [p]$ in the objective map $H$, and denote $L := \max\{ L^{j} \mid j \in [p] \}$. Let $x^{k}$ be a non-stationary point and $\alpha_{k}$ be a step-size generated by \emph{Algorithm \ref{algo 1}} at $x^{k}$. Then, for $k\in \mathbb{N}\cup\{0\}$,
    \begin{equation}\label{step_size-bound}
        \alpha_{k}\geq \min\left\{\frac{(1-\rho)\mu l}{L},1\right\}.
    \end{equation}
\end{lemma}

\begin{proof}
    Note that, for $\alpha_k=1$, the relation \eqref{step_size-bound} is trivial. Now, consider $\alpha_{k}<1$. Therefore, from Steps 0 and 4 of Algorithm \ref{algo 1}, it follows that there exists $\hat{j}\in [\omega_{k}]$ such that
    \begin{align*}
       & h^{a^{l}_{k,\hat{j}}}\left(x^{k}+\tfrac{\alpha_{k}}{\mu} v_{l}^{k}\right)- h^{a^{l}_{k,\hat{j}}}\left(x^{k}\right)-\rho \tfrac{\alpha_{k}}{\mu}\left(\nabla f^{a^{l}_{k,\hat{j}}}\left(x^{k}\right)^{\top}v_{l}^{k}+g^{a^{l}_{k,\hat{j}}}\left(x^{k}+v_{l}^{k}\right)- g^{a^{l}_{k,\hat{j}}}\left(x^{k}\right)\right)\npreceq_{C} 0\\
     \implies &    \left\{f^{a^{l}_{k,\hat{j}}}\left(x^{k}+\tfrac{\alpha_{k}}{\mu} v_{l}^{k}\right)- f^{a^{l}_{k,\hat{j}}}\left(x^{k}\right) -\tfrac{\alpha_{k}}{\mu}\nabla f^{a^{l}_{k,\hat{j}}}\left(x^{k}\right)^{\top}v_{l}^{k} \right\} \\
     & \hspace{3cm}+  \left\{ \tfrac{\alpha_{k}}{\mu}\nabla f^{a^{l}_{k,\hat{j}}}\left(x^{k}\right)^{\top}v_{l}^{k} +g^{a^{l}_{k,\hat{j}}}\left(x^{k}+\tfrac{\alpha_{k}}{\mu} v_{l}^{k}\right)- g^{a^{l}_{k,\hat{j}}}\left(x^{k}\right)  \right\}\\
     &  -\rho\left\{ \tfrac{\alpha_{k}}{\mu}\left(\nabla f^{a^{l}_{k,\hat{j}}}\left(x^{k}\right)^{\top}v_{l}^{k}+g^{a^{l}_{k,\hat{j}}}\left(x^{k}+v_{l}^{k}\right)- g^{a^{l}_{k,\hat{j}}}\left(x^{k}\right)\right)\right\}\npreceq_{C} 0\\
     \overset{\eqref{convex_direction_deri}}{\implies}  &  \left\{f^{a^{l}_{k,\hat{j}}}\left(x^{k}+\tfrac{\alpha_{k}}{\mu} v_{l}^{k}\right)- f^{a^{l}_{k,\hat{j}}}\left(x^{k}\right) -\tfrac{\alpha_{k}}{\mu}\nabla f^{a^{l}_{k,\hat{j}}}\left(x^{k}\right)^{\top}v_{l}^{k} \right\}\\
     & \hspace{3cm}+ \left\{\tfrac{\alpha_{k}}{\mu}(1-\rho)\left(\nabla f^{a^{l}_{k,\hat{j}}}\left(x^{k}\right)^{\top}v_{l}^{k}+g^{a^{l}_{k,\hat{j}}}\left(x^{k}+v_{l}^{k}\right)- g^{a^{l}_{k,\hat{j}}}\left(x^{k}\right)\right)\right\}\npreceq_{C} 0.
    \end{align*}
From Lemma \ref{scal_func}(ii) and (i), we have
\begin{align*}
    \phi \Big(&  \left\{f^{a^{l}_{k,\hat{j}}}\left(x^{k}+\tfrac{\alpha_{k}}{\mu} v_{l}^{k}\right)- f^{a^{l}_{k,\hat{j}}}\left(x^{k}\right) -\tfrac{\alpha_{k}}{\mu}\nabla f^{a^{l}_{k,\hat{j}}}\left(x^{k}\right)^{\top}v_{l}^{k} \right\}\\
     & \hspace{3cm}+ \left\{\tfrac{\alpha_{k}}{\mu}(1-\rho)\left(\nabla f^{a^{l}_{k,\hat{j}}}\left(x^{k}\right)^{\top}v_{l}^{k}+g^{a^{l}_{k,\hat{j}}}\left(x^{k}+v_{l}^{k}\right)- g^{a^{l}_{k,\hat{j}}}\left(x^{k}\right)\right)\right\}\Big)>0\\
     \implies \phi&  \left(f^{a^{l}_{k,\hat{j}}}\left(x^{k}+\tfrac{\alpha_{k}}{\mu} v_{l}^{k}\right)- f^{a^{l}_{k,\hat{j}}}\left(x^{k}\right) -\tfrac{\alpha_{k}}{\mu}\nabla f^{a^{l}_{k,\hat{j}}}\left(x^{k}\right)^{\top}v_{l}^{k} \right)\\
     & \hspace{3cm}+ \tfrac{\alpha_{k}}{\mu}(1-\rho)\phi\left(\nabla f^{a^{l}_{k,\hat{j}}}\left(x^{k}\right)^{\top}v_{l}^{k}+g^{a^{l}_{k,\hat{j}}}\left(x^{k}+v_{l}^{k}\right)- g^{a^{l}_{k,\hat{j}}}\left(x^{k}\right)\right)>0,
\end{align*}
 because $\alpha_{k}>0$, $\mu>0$, and $\rho <1$.  Since $\nabla f^{j}$ is Lipschitz continuous for each $j \in [p]$ with Lipschitz constant $L^{j}$, and $L := \max\{ L^{j} \mid j \in [p] \}$, it follows from Lemma~\ref{descent lemma} that
 \[ \tfrac{L}{2}\left \|\tfrac{\alpha_{k}}{\mu}v_{l}^{k}\right\|^{2} +\frac{\alpha_{k}}{\mu}(1-\rho)\phi\left(\nabla f^{a^{l}_{k,\hat{j}}}\left(x^{k}\right)^{\top}v_{l}^{k}+g^{a^{l}_{k,\hat{j}}}\left(x^{k}+v_{l}^{k}\right)- g^{a^{l}_{k,\hat{j}}}\left(x^{k}\right)\right)>0. \]
Since $\hat{j}\in [\omega_{k}]$, the following relation holds
 \[ \tfrac{L}{2}\left \|\tfrac{\alpha_{k}}{\mu}v_{l}^{k}\right\|^{2} +\frac{\alpha_{k}}{\mu}(1-\rho)\underset{j\in [\omega_{k}]}{\max}\phi\left(\nabla f^{a^{l}_{k,j}}\left(x^{k}\right)^{\top}v_{l}^{k}+g^{a^{l}_{k,j}}\left(x^{k}+v_{l}^{k}\right)- g^{a^{l}_{k,j}}\left(x^{k}\right)\right)>0. \]
 From the relation \eqref{R_eqv_I}, it follows that
 \begin{align*}
     \tfrac{L}{2}\left \|\tfrac{\alpha_{k}}{\mu}v_{l}^{k}\right\|^{2} +\frac{\alpha_{k}}{\mu}(1-\rho)\left(-\tfrac{l}{2}\left\|v^{k}_{l}\right\|^{2}\right)>0
     \implies  \alpha_{k}> \tfrac{(1-\rho)\mu l}{L}.
 \end{align*}
 This concludes the proof.
\end{proof}

\begin{theorem}\label{T_com_ana_1}
     Consider the problem \eqref{sop}, where the map $H$ satisfies \emph{Assumption~\ref{Ass_label_bdd}}. Assume that $\nabla f^{j}$ is Lipschitz continuous with Lipschitz constant $L^{j}$ for every $j \in [p]$ in the objective map $H$, and denote $L := \max\{ L^{j} \mid j \in [p] \}$. Let $\{x^{k}\}$ be a sequence of non-stationary points generated by \emph{Algorithm \ref{algo 1}}. Then, for any $\hat{k} \in \mathbb{N}$, there exists an iteration $k \in \{0,1,\dots,\hat{k}\}$ such that 
     \[
\lvert\Theta_{l}\left(x^{k}\right)\rvert<  \frac{\mathcal{C}_{1}}{\hat{k}},
\]
where $\mathcal{C}_{1}>0$ is a constant.
\end{theorem}

\begin{proof}
   Note that $x^{k}\in \mathrm{dom}(H)$ and is a non-stationary point. Therefore, from Step 4 of Algorithm \ref{algo 1}, for all $j\in [\omega_{k}]$, we get the following
   \begin{align*}
      &  h^{a^{l}_{k,j}}\left(x^{k+1}\right)-h^{a^{l}_{k,j}}\left(x^{k}\right)\preceq_{C} \rho \,\alpha_{k}\left(\nabla f^{a^{l}_{k,j}}\left(x^{k}\right)^{\top}v_{l}^{k}+g^{a^{l}_{k,j}}\left(x^{k}+v_{l}^{k}\right)- g^{a^{l}_{k,j}}\left(x^{k}\right)\right)\\
      \implies & \phi \left(h^{a^{l}_{k,j}}\left(x^{k+1}\right)-h^{a^{l}_{k,j}}\left(x^{k}\right)\right)\leq  \rho \,\alpha_{k}\phi \left(\nabla f^{a^{l}_{k,j}}\left(x^{k}\right)^{\top}v_{l}^{k}+g^{a^{l}_{k,j}}\left(x^{k}+v_{l}^{k}\right)- g^{a^{l}_{k,j}}\left(x^{k}\right)\right) 
   \end{align*}
   because of Lemma \ref{scal_func}(ii) and (i). Therefore, it follows that
   \begin{align*}
     & \rho \,\alpha_{k} \underset{j\in [\omega_{k}]}{\max}\phi \left(\nabla f^{a^{l}_{k,j}}\left(x^{k}\right)^{\top}v_{l}^{k}+g^{a^{l}_{k,j}}\left(x^{k}+v_{l}^{k}\right)- g^{a^{l}_{k,j}}\left(x^{k}\right)\right)\\
     \geq &~ \underset{j\in [\omega_{k}]}{\max}\phi \left(h^{a^{l}_{k,j}}\left(x^{k+1}\right)-h^{a^{l}_{k,j}}\left(x^{k}\right)\right) \geq ~\underset{j\in [\omega_{k}]}{\max}\phi \left(h^{a^{l}_{k,j}}\left(x^{k+1}\right)\right)-\underset{j\in [\omega_{k}]}{\min}\phi \left(h^{a^{l}_{k,j}}\left(x^{k}\right)\right)
   \end{align*}
   from Lemma \ref{scal_func} (i). Accordingly, applying the function $\xi$ defined in~\eqref{zeta} and the equation \eqref{Mo_zeta}, we get
   \begin{align*}
        &   \rho\, \alpha_{k} \underset{j\in [\omega_{k}]}{\max}\phi \left(\nabla f^{a^{l}_{k,j}}\left(x^{k}\right)^{\top}v_{l}^{k}+g^{a^{l}_{k,j}}\left(x^{k}+v_{l}^{k}\right)- g^{a^{l}_{k,j}}\left(x^{k}\right)\right)\\
        \geq&
        ~ \underset{j\in [\omega_{k}]}{\max}\phi \left(h^{a^{l}_{k,j}}\left(x^{k+1}\right)\right)-\xi \circ H\left(x^{k}\right).
        \end{align*}
        This implies that 
    \[\underset{j\in [\omega_{k}]}{\max}\phi \left(\nabla f^{a^{l}_{k,j}}\left(x^{k}\right)^{\top}v_{l}^{k}+g^{a^{l}_{k,j}}\left(x^{k}+v_{l}^{k}\right)- g^{a^{l}_{k,j}}\left(x^{k}\right)\right)\geq \tfrac{1}{\rho\,\alpha_{k}}\left\{\xi \circ H\left(x^{k+1}\right)-\xi \circ H\left(x^{k}\right)\right\}.\] 
  Since $x^{k}$ is a non-stationary point, by the definition of $\Theta_{l}$, we obtain
   \begin{align*}
     & \Theta_{l}\left(x^{k}\right)> \tfrac{1}{\rho\,\alpha_{k}}\left\{\xi \circ H\left(x^{k+1}\right)-\xi \circ H\left(x^{k}\right)\right\}\\
     \overset{\eqref{step_size-bound}}{\implies} &  \Theta_{l}\left(x^{k}\right)> \tfrac{1}{\rho~\alpha_{\mathrm{min}}}\left\{\xi \circ H\left(x^{k+1}\right)-\xi \circ H\left(x^{k}\right)\right\} ,
   \end{align*}
  where $\alpha_{\mathrm{min}}:=\min\left\{\tfrac{(1-\rho)\mu l}{L},1\right\}.$ Accordingly, we have
   \begin{align*}
&\sum_{k=0}^{\hat{k}-1}\Theta_{l}\left(x^{k}\right) \geq  \tfrac{1}{\rho~\alpha_{\mathrm{min}}}\sum_{k=0}^{\hat{k}-1}\left\{\xi \circ H\left(x^{k+1}\right)-\xi \circ H\left(x^{k}\right)\right\}=\tfrac{1}{\rho\,\alpha_{\mathrm{min}}}\left\{\xi \circ H\left(x^{\hat{k}}\right)-\xi \circ H\left(x^{0}\right)\right\}\\
   \implies & \hat{k}\underset{0\leq k \leq \hat{k}-1}{\max}\Theta_{l}\left(x^{k}\right)\geq \tfrac{1}{\rho~\alpha_{\mathrm{min}}}\left\{\xi \circ H\left(x^{\hat{k}}\right)-\xi \circ H\left(x^{0}\right)\right\}\overset{\eqref{lower_bdd_H}}{\geq} \tfrac{1}{\rho\,\alpha_{\mathrm{min}}}\left\{\xi \circ S-\xi \circ H\left(x^{0}\right)\right\}\\
    \implies & \underset{0\leq k \leq \hat{k}-1}{\max}\Theta_{l}\left(x^{k}\right)\geq -\frac{\mathcal{C}_{1}}{\hat{k}},   
   \end{align*}
 where $\mathcal{C}_{1}:=\tfrac{1}{\rho\,\alpha_{\mathrm{min}}}\left\{\xi \circ H\left(x^{0}\right)-\xi \circ S\right\}>0$. Consequently, there exists $k \in \{0,1,\dots,\hat{k}\}$ for which 
$\lvert\Theta_{l}\left(x^{k}\right)\rvert < \frac{\mathcal{C}_{1}}{\hat{k}}$ holds.
\end{proof}

We now present the complexity analysis of Algorithm~\ref{algo} in the following result.\\

\begin{theorem}\label{T_com_ana_2}
   Consider the problem~\eqref{sop}, where the map $H$ satisfies \emph{Assumption~\ref{Ass_label_bdd}}.  
Assume that $\nabla f^{j}$ is Lipschitz continuous with Lipschitz constant $L^{j}$ for every $j \in [p]$, and denote $L := \max\{ L^{j} \mid j \in [p] \}$.  
Let the scalar $l\geq L$ in $\Theta_{l}$ as defined in~\eqref{opt_val} and $\{x^{k}\}$ be a sequence of non-stationary points generated by \emph{Algorithm~\ref{algo}}.  
Then, for any $\hat{k} \in \mathbb{N}$, there exists some $k \in \{0,1,\dots,\hat{k}\}$ satisfying  
$$\lvert\Theta_{l}\left(x^{k}\right)\rvert < \frac{\mathcal{C}_{2}}{\hat{k}},$$
where $\mathcal{C}_{2}>0$ is a constant.
\end{theorem}

\begin{proof}
    In view of Lemma \ref{scal_func}(i), for any $j\in [p]$, we have
    \begin{align*}
        \phi\left(h^{j}\left(x^{k+1}\right)-h^{j}\left(x^{k}\right)\right)&= \phi\Big(\left\{f^{j}\left(x^{k+1}\right)-f^{j}\left(x^{k}\right)- \nabla f^{j}\left(x^{k}\right)^{\top}(x^{k+1}-x^{k})\right\} \\ 
        &\hspace{2.5cm} + \left\{ \nabla f^{j}\left(x^{k}\right)^{\top}(x^{k+1}-x^{k})+ g^{j}\left(x^{k+1}\right)-g^{j}\left(x^{k}\right) \right\}\Big)\\
        &\leq \phi\left(f^{j}\left(x^{k+1}\right)-f^{j}\left(x^{k}\right)- \nabla f^{j}\left(x^{k}\right)^{\top}(x^{k+1}-x^{k})\right) \\ 
        &\hspace{2.5cm} + \phi
        \left( \nabla f^{j}\left(x^{k}\right)^{\top}(x^{k+1}-x^{k})+ g^{j}\left(x^{k+1}\right)-g^{j}\left(x^{k}\right) \right)\\
           &\leq \phi\left(\nabla f^{j}\left(x^{k}\right)^{\top}(x^{k+1}-x^{k})+g^{j}\left(x^{k+1}\right)-g^{j}\left(x^{k}\right) \right)+\frac{L}{2}\|x^{k+1}-x^{k} \|^{2}
    \end{align*}
    because of Lemma \ref{descent lemma}. Accordingly, for the partition element $a_{k}\in P_{k}$, it follows that
    \begin{align*}
      &\underset{j\in [\omega_{k}]}{\max}\,\phi\left(h^{a^{l}_{k,j}}\left(x^{k+1}\right)-h^{a^{l}_{k,j}}\left(x^{k}\right)\right)\\
      \leq &~ \underset{j\in [\omega_k]}{\max}\,\phi\left(\nabla f^{a^{l}_{k,j}}\left(x^{k}\right)^{\top}(x^{k+1}-x^{k})+g^{a^{l}_{k,j}}\left(x^{k+1}\right)-g^{a^{l}_{k,j}}\left(x^{k}\right)\right) +\frac{L}{2}\|x^{k+1}-x^{k} \|^{2}\\
      \leq &~ \Theta_{l}\left(x^{k}\right)
    \end{align*}
    from the definition of $\Theta_{l}$ with the fact that $l\geq L$. Therefore, by Lemma~\ref{scal_func}(i), together with the definition of the function $\xi$ in~\eqref{zeta} and the relation in~\eqref{Mo_zeta}, we obtain

    \begin{align*}
    &\Theta_{l}\left(x^{k}\right)\geq \underset{j\in [\omega_{k}]}{\max}\,\phi(h^{a^{l}_{k,j}}\left(x^{k+1}\right))-\underset{j\in [\omega_{k}]}{\min}~\phi(h^{a^{l}_{k,j}}\left(x^{k}\right))\geq \xi \circ H\left(x^{k+1}\right)-\xi \circ H\left(x^{k}\right)\\
     \implies & \sum_{k=0}^{\hat{k}-1}\Theta_{l}\left(x^{k}\right) \geq \sum_{k=0}^{\hat{k}-1}\left\{\xi \circ H\left(x^{k+1}\right)-\xi \circ H\left(x^{k}\right)\right\}= \xi \circ H\left(x^{\hat{k}}\right)-\xi \circ H\left(x^{0}\right)\\
        \implies & \hat{k}\underset{0\leq k\leq \hat{k}-1}{\max}\Theta_{l}\left(x^{k}\right) \geq \xi \circ H\left(x^{\hat{k}}\right)-\xi \circ H\left(x^{0}\right)\overset{\eqref{lower_bdd_H}}{\geq} \xi \circ S-\xi \circ H\left(x^{0}\right) \\
        \implies & \underset{0\leq k\leq \hat{k}-1}{\max}\Theta_{l}\left(x^{k}\right)\geq -\frac{\mathcal{C}_{2}}{\hat{k}},
    \end{align*} 
   where $\mathcal{C}_{2}:=\xi \circ H\left(x^{0}\right)-\xi \circ S>0$. This implies that for some $k \in \{0,1,\ldots,\hat{k}\}$, we have  
$\lvert\Theta_{l}\left(x^{k}\right)\rvert \leq \frac{\mathcal{C}_{2}}{ \hat{k}}$.

\end{proof}

Note that convergence rate of the sequence $\{\|v^{k}_{l}\|\}$
is used as a rate of convergence of proximal-type methods. Accordingly, from Theorems~\ref{T_com_ana_1} and~\ref{T_com_ana_2}, the convergence rate of the proposed algorithms with certain commonly used assumptions is $\mathcal{O}(1/\sqrt{k})$.

\section{Numerical results and illustrations}\label{sec.6}  

In this section, we present the numerical performance of the proposed algorithms, viz. Algorithm~\ref{algo 1} (abbreviated as Algo \ref{algo 1}) and Algorithm~\ref{algo} (abbreviated as Algo~\ref{algo}), to evaluate the practical efficiency and assess their relative performance. All performances are conducted in MATLAB R2023b on a Laptop equipped with an Intel Core i5-11320H CPU (3.20 GHz), 8 GB RAM, and the Windows 11 operating system. In the following, we first discuss the choice of objective functions for the problem~\eqref{sop}, which are used to evaluate the numerical performance. We then provide implementation details for both methods, and finally, we present and analyze the numerical results.

\begin{framed}
\noindent
Readers may think that ``the numerical comparisons of the proposed methods should be done with the current state-of-the-art for set optimization methods.'' To answer this question, it is noteworthy to observe that in the existing literature, there is no numerical method to solve a non-smooth set optimization problem such as \eqref{sop}. Accordingly, we have restricted our comparisons to their relative performance.

\end{framed}

\subsection{Problem formulation}
The objective function $H$ of the problem~\eqref{sop} is constructed by using the sum of two functions, $f^{j}+g^{j}$, where $f^{j}$ is continuously differentiable with $\nabla f^{j}$ is Lipschitz continuous and $g^{j}$ is a proper, closed, and convex function, for each $j\in [p]$. For the choice of the function $f^{j}$, we adopt two approaches. In the first approach, $f^{j}$ is taken directly from some test problems for set optimization given in Table \ref{Test_Problems}. In the second approach, we start with the objective function $\mathscr{F}$ from a vector optimization test problem given in Table \ref{Test_Problems} and define a collection of functions $F^{j}:\mathbb{R}^{n}\to \mathbb{R}^{m}$, $j\in [p]$, then combine them as $f^{j}(x):=\mathscr{F}(x)+F^{j}(x)$.


\begin{table}[ht]
\centering
\begin{tabular}{ l c l l l l }
\toprule  
 {Problem} & {Reference } &  {$m$} &  {$n$} &  {$lb^{\top}$} &  {$ub^{\top}$} \\ 
 \midrule 
AP1        &  \cite{AP-2015}      &3      & 2       &   $(-10,-10)$         & $(10,10)$         \\ 
AP3        &  \cite{AP-2015}      &2      &2        &   $(-10,-10)$         & $(10,10)$         \\ 
AP4        &   \cite{AP-2015}     &3      &3        &  $(-10,-10,-10)$      & $(10,10,10)$         \\ 
BK1        &  \cite{NE_36}      & 2     & 2       & $(-5,-5)$             & $(10,10)$         \\ 
DD1        &  \cite{DD-1998}      & 2     & 5       & $(-20,-20,\ldots,-20)$ & $(20,20,\ldots,20)$          \\ 
DTLZ1      &   \cite{DTLZ-2005}     & 3     &  5      &   $(0,0,\ldots,0)$    & $(1,1,\ldots,1)$         \\ 
FDS        & \cite{NE_20}       & 3     &  5      &   $(-2,-2)$           & $(2,2)$         \\ 
GAAZ7       & \cite{ghosh2024newton} & 2   & 2   & $(-\pi,-\pi)$        & $(\pi,\pi)$           \\ 
GEC1       & \cite{steepest2021set_optimization}      & 2     & 1       &   $-5\pi$             & $5\pi$         \\ 
GEC2       &  \cite{steepest2021set_optimization}      & 3     & 2       &   $(-50,-50)$          & $(50,50)$         \\ 
GEC3       &   \cite{steepest2021set_optimization}     & 2     & 2       &   $(-10\pi,-10\pi)$    & $(10\pi,10\pi)$         \\
GRPY2      & \cite{Ghosh CGM} & 2   & 2   & $(-\pi,-\pi)$      & $(\pi,\pi)$ \\ 

Hil1       &   \cite{NE_34}     & 2      & 2        &   $(0,0)$               & $(1,1)$         \\ 
IKK1       &   \cite{NE_36}     & 3      & 2        &   $(-50,-50)$           & $(50,50)$         \\ 
JOS1       &   \cite{NE_38}     & 2      & 10,~100  &   $(-2,-2,\ldots,-2)$    & $(2,2,\ldots,2)$         \\
KW2        &  \cite{KW-2005}      & 2      & 2        &   $(-3,-3)$              & $(3,3)$         \\  
Lov5       &  \cite{Lov-2011}      & 2      & 3        &   $(-2,-2)$              & $(2,2)$         \\  
MOP1       &  \cite{NE_36}      & 2      & 1        &   $-10^{5}$              & $10^{5}$         \\ 
MOP7       & \cite{NE_36}       & 3      & 2        &   $(-400,-400)$           & $(400,400)$         \\ 
PNR        &  \cite{PNR-2006}      & 2      & 2        &   $(-2,-2)$               & $(2,2)$         \\ 
SD         &  \cite{SD_1992}      & 2      & 4        &   $(1,\sqrt{2},\sqrt{2},1)$     & $(3,3,3,3)$         \\ 
Toi4       & \cite{Toint1983}       & 2      & 4        &   $(-2,-2,\ldots,-2)$     & $(5,5,\ldots,5)$         \\ 
TRIDIA     & \cite{Toint1983}       & 3      & 3        &   $(-1,-1,-1)$            & $(1,1,1)$         \\ 
VU2        &    \cite{NE_36}    & 2      & 2        &   $(-3,-3)$               & $(3,3)$         \\ 
 \bottomrule
\end{tabular}
\caption{Test Problems}\label{Test_Problems} 
\end{table}

In the following, we define two collections of functions. 

\noindent
\textbf{SVM1:} For each $j \in [p]$, define $F^{j}:\mathbb{R}^{n} \to \mathbb{R}^{2}$  
as 
$F^{j}(x):=\left(F^{j}_{1}(x),\, F^{j}_{2}(x)\right)^{\top}$
with
\begin{align*}
   & F^{j}_{1}(x):=\sum_{k=0}^{n-1}\tfrac{1}{2^{k}}\left(\cos\!\left(x+\tfrac{2\pi(j-1)}{p}\right) 
+ \sin\!\left(x+\tfrac{2\pi(j-1)}{p}\right)\right),\\
 & F^{j}_{2}(x):=\cos^{2}\!\left(\sum_{i=1}^{n}x_{i}+\tfrac{2\pi(j-1)}{p}\right).
\end{align*}

\noindent 
\textbf{SVM2:} Consider, for each $j \in [p]$, the function $F^{j}:\mathbb{R}^{n} \to \mathbb{R}^{3}$ defined as $F^{j}(x):=\left(F^{j}_{1}(x),\, F^{j}_{2}(x),\, F^{j}_{3}(x)\right)^{\top}$
with
\begin{align*}
F^{j}_{1}(x)&:=\left(1+0.25\sin\!\left(\sum_{i=1}^{n}x_{i}\right)\right)\cos\!\left(\tfrac{2\pi(j-1)}{p}\right), \\[0.5em]
F^{j}_{2}(x)&:=\left(1+0.25\sin\!\left(\sum_{i=1}^{n}x_{i}\right)\right)\sin\!\left(\tfrac{2\pi(j-1)}{p}\right), \\[0.5em]
F^{j}_{3}(x)&:=0.25\cos\!\left(\tfrac{2\pi(j-1)}{p}+0.5\sum_{i=1}^{n}x_{i}\right).
\end{align*}

In view of the function $g^{j},~j\in [p]$, we incorporate uncertainty in each component of $g^{j}$ through a scalar multiple in the component of the decision variable and construct its pessimistic counterpart via the robust minimax formulation. Specifically, we define $g^{j}:\mathbb{R}^{n}\to \mathbb{R}^{m}$ as follows:
\[
g^{j}(x):=\left(g_{1}^{j}(x),~ g_{2}^{j}(x),\ldots,~ g_{m}^{j}(x)\right)^{\top},
\]
where
\[
g_{i}^{j}(x):=\max_{y\in \mathcal{Y}^{j}_{i}} x^{\top}y, \quad 
\mathcal{Y}^{j}_{i}:=\left\{y\in\mathbb{R}^{n} ~\middle|~ -\delta e \leq A^{j}_{i}y \leq \delta e\right\}
\]
with $\delta>0$, $e=(1,1,\ldots,1)^{\top}\in \mathbb{R}^{n}$, and $A^{j}_{i}\in \mathbb{R}^{n \times n}$ a non-singular matrix. 

Note that the functions in the set-valued maps SVM1 and SVM2, as well as the functions $g^{j}$, $j\in [p]$, are all defined on $\mathbb{R}^{n}$.  
Accordingly, we assume that $\mathrm{dom}(H)=[lb,\,ub]$, where $lb$ and $ub$ denote the lower and upper bounds of the decision variable $x$ in the associated map or function used to construct the objective map $H$, as given in Table~\ref{Test_Problems}.  
For simplicity, each problem is referred to by the name of its associated problem listed in Table~\ref{Test_Problems}.

In view of the choices of ordering cones, we adopt two cones as given below:
\[C_{1}:=\left\{(w_{1}, w_{2}, \ldots, w_{m})^{\top}~\middle |~ w_{i}\geq 0 ~\forall i\in [m]\right\} \]
and 
\[C_{2}:=\left\{(w_{1}, w_{2})^{\top}~\middle |~ 10w_{1}-\tfrac{1}{10}w_{2}\geq 0, ~-\tfrac{1}{10}w_{1}+10w_{2}\geq 0 \right\}. \]

\subsection{Implementation details}
In the following, we specify the parameter settings and clarify certain terms required for executing each step of both algorithms. Furthermore, we outline the experimental setup employed to facilitate a comparison between the methods.

\begin{itemize}
    \item We set the parameters $\rho=10^{-4}$ and $\mu=0.5$. The value of $\delta$ is chosen arbitrarily in the range $[0.01,~0.1]$. For each $i\in [m]$ and $j\in [p]$, we choose an arbitrary non-singular matrix $A^{j}_{i}\in \mathbb{R}^{n\times n}$.
\\

    \item To evaluate $g^{j}(x)$ for $j\in [p]$, we reformulate each component $g^{j}_{i}(x)$ as a linear programming problem by considering its dual formulation as follows: consider a primal problem \[\begin{aligned}
        & \max && x^{\top} y\\
        & \text{subject to} && \Tilde{A}^{j}_{i}y\leq \Tilde{b},~y\in \mathbb{R}^{n},
    \end{aligned}\]
    where $\Tilde{A}^{j}_{i}:=\left[A^{j}_{i};~ -A^{j}_{i}\right]\in \mathbb{R}^{2n\times n}$ and $\Tilde{b}:=[\delta e;~-\delta e]\in \mathbb{R}^{2n}$. Then, a corresponding dual problem is given by
    \begin{equation}\label{dual_prob}
         \begin{aligned}
        & \min && \Tilde{b}^{\top} w\\
        & \text{subject to} && \Tilde{A}^{\top}w= x, ~w\in \mathbb{R}_{+}^{2n}.\\
    \end{aligned}
    \end{equation}
    We use Matlab solver \emph{linprog} to solve the problem \eqref{dual_prob}, which enables efficient computation of $g^{j}(x)$.

    \item In order to find the minimal set $\mathcal{M}(H(\bar{y}),C)$, the pair-wise compression of the set $H(\bar{y})$ is used corresponding to the cone $C$.\\

    \item  We fix value of $l=1$ for Algo \ref{algo 1}. For Algo \ref{algo}, where the Lipschitz constant $L$ cannot be easily estimated, the following strategy is adopted. We initialize $l:=1$ and check whether the following relation is satisfied:
    \begin{equation}\label{lip_esti}
      \underset{{j\in [\omega_{k}]}}{\max}  \phi\left(h^{a_{k,j}}\left(x^{k}+v^{k}_{l}\right) - h^{a_{k,j}}\left(x^{k}\right)\right) \leq \Theta_{l}\left(x^{k}\right).  
    \end{equation}
If the relation holds, we accept the current value of $l$; otherwise, we update $l := 2l$ and recheck the relation \eqref{lip_esti}.  
Since $L$ is finite, therefore, in view of the relation \eqref{Prp4.1_eq1}, this procedure is guaranteed to terminate after finitely many steps.
\\

    \item We compute $a^{l}_{k}$ and $v^{k}_{l}$ at the iterate $x^{k}$ as follows. First, for each $a\in P_{k}$, we solve the following problem:
    \begin{equation}\label{sub_prob_a}
        \underset{v\in  \mathbb{R}^{n}}{\min}\left\{\underset{j\in [\omega_k]}{\max}\,\phi\left(\nabla f^{a_{j}}\left(x^{k}\right)^{\top}v+g^{a_{j}}(x^{k}+v)- g^{a_{j}}\left(x^{k}\right)\right)+\frac{l}{2}\|v\|^2\right\}.
    \end{equation}
    Then, we take the minimum over $a\in P_{k}$ and select a minimizer denoted by $a^{l}_{k}$ and the corresponding solution of \eqref{sub_prob_a} is denoted by $v^{k}_{l}$.
 In order to address the problem~\eqref{sub_prob_a}, we scalarize the non-quadratic terms as follows:
    \[\begin{aligned}
        & \min~&&t+\frac{l}{2}\|v\|^{2}\\
        & \text{subject to} && \phi\left(\nabla f^{a_{j}}\left(x^{k}\right)^{\top}v+g^{a_{j}}(x^{k}+v)- g^{a_{j}}\left(x^{k}\right)\right)\leq t \quad \forall j\in [\omega_{k}],\\
        & & & (t,v)\in \mathbb{R}\times \mathbb{R}^{n}.
    \end{aligned}\]
Thereafter, by employing the dual formulation \eqref{dual_prob}, the scalarized problem is reformulated as a quadratic program. We then use Matlab solver \emph{quadprog} to solve this problem. For more details, see \cite[Section 5.2]{tanabe2019proximal}.\\

    \item For the stopping condition of Algorithms~\ref{algo 1} and~\ref{algo}, we use the condition $\lvert\Theta_{l}\left(x^{k}\right)\rvert<10^{-5}$.\\

    \item For each considered problem, we arbitrarily select 100 initial points within the domain $[lb,\,ub]$ specified in Table~\ref{Test_Problems}. The proposed methods are then executed on these same points, using the same choice of $A$ and $b$ for each initial point, to ensure a fair comparison.\\

    \item We employ three metrics to evaluate the performance of the proposed methods:
\begin{itemize}
    \item Iteration count: the number of iterations of the sequence $\{x^{k}\}$, starting from an initial point $x^{0}$, required to satisfy the prescribed stopping condition.\\
    
    \item Computational time: the time taken to reach the stopping condition from a given initial point $x^{0}$.\\
    
    \item Solve: out of 100 arbitrarily chosen initial points, the number that satisfy the stopping condition within 500 iterations. If, for any initial point, a solver fails to find a solution within this limit, the corresponding data on iteration count and time are excluded for that solver, and the attempt is considered unsuccessful.
\end{itemize}


\end{itemize}

\subsection{Performance}
We now compare our proposed methods on formulated test problems associated with the test problem given in Table \ref{Test_Problems}. Here, we have two SVMs; we use a compatible one of them if required. We abbreviate minimum and maximum as min and max, respectively. The comparison is given in Table \ref{performance_table} and demonstrated by the performance profile graph in Fig. \ref{performance_profiles_graph}. Thereafter,
we illustrate the trajectory of $\{H\left(x^{k}\right)\}$ starting from an arbitrarily chosen point for some of the test problems using the better-performing method in Fig. \ref{Grap}. In addition, we also present the trajectories of $\left\{f^{j}\left(x^{k}\right)~\middle|~j\in [p]\right\}$ corresponding to the trajectory of $\{H\left(x^{k}\right)\}$, which highlights the impact of uncertainty.

\begin{landscape} 

\begin{table}[ht!]
\begin{center}
\resizebox{22cm}{!}{ 
\begin{tabular}{l l l l l l l l l l l l l l l l l l}
\toprule
\multicolumn{1}{c}{Problem} & \multicolumn{1}{c}{SVM} & \multicolumn{1}{c}{$C$} & \multicolumn{1}{c}{$p$} & \multicolumn{2}{c}{Solved}  &\multicolumn{2}{c}{\text{Min Iteration}} & \multicolumn{2}{c}{\text{Mean Iteration}} & \multicolumn{2}{c}{\text{max Iteration}} & \multicolumn{2}{c}{\text{Min time}} & \multicolumn{2}{c}{\text{Mean time}} & \multicolumn{2}{c}{\text{max time}}\\
\cmidrule(rl){5-6}\cmidrule(rl){7-8} \cmidrule(rl){9-10} \cmidrule(rl){11-12} \cmidrule(rl){13-14} \cmidrule(rl){15-16} \cmidrule(rl){17-18}
 &  &  &  & Algo \ref{algo 1} & Algo \ref{algo}& Algo \ref{algo 1} & Algo \ref{algo} & Algo \ref{algo 1} & Algo \ref{algo}& Algo \ref{algo 1} & Algo \ref{algo}& Algo \ref{algo 1} & Algo \ref{algo}& Algo \ref{algo 1} & Algo \ref{algo} & Algo \ref{algo 1} & Algo \ref{algo} \\
\midrule
AP1 & 2 & $C_{1}$ & 50 & 100 & 100 & 0 & 0 & 11.4800 & 15.4400 & 142 & 106 & 0.9222 & 3.4590 & 47.8762 & 70.4389 & 585.8116 & 422.3846 \\

AP3 & 1 & $C_{1}$ & 50 & 100 & 100 & 1 & 5 & 14.7600 & 17.1200 & 85 & 62 & 0.3777 & 1.4290 & 3.1697 & 4.5954 & 17.0800 & 18.5134\\

AP4 & 2 & $C_{1}$ & 50 & 100 & 100 & 0 & 0 & 32.5000 & 39.1000 & 281 & 168 & 1.1000 & 1.1000 & 84.7000 & 87.8000 & 1006.4000 & 385.2000\\

BK1 & 1 & $C_{1}$ & 50 & 100 & 100 & 1 & 1 & 2.3800 & 7.1800 & 7 & 12 & 1.3357 & 1.7345 & 2.6097 & 6.7350 & 6.2383 & 11.9521\\

BK1 & 1 & $C_{2}$ & 50 & 84 & 100 & 0 & 0 & 2.4048 & 1.4800 & 4 & 12 & 0.6473 & 4.4012 & 5.8381 & 14.1328 & 12.3170 & 81.1260\\

DD1 & 1 & $C_{1}$ & 50 & 100 & 100 & 13 & 16 & 58.1800 & 93.8800 & 148 & 220 & 2.4979 & 3.2616 & 11.1108 & 19.9241 & 27.0177 & 49.1342\\

DTLZ1 & 2 & $C_{1}$ & 2 & 74 & 100 & 14 & 1 & 79.8108 & 8.0600 & 463 & 27 & 3.2932 & 0.4538 & 44.3037 & 3.3752 & 332.4419 & 6.6795 \\

FDS & 2 & $C_{1}$ & 5 & 100 & 100 & 18 & 9 & 72.2200 & 22.6200 & 203 & 36 & 5.3832 & 1.3165 & 33.9300 & 4.0184 & 130.6716 & 8.3389\\

GAAZ7 & - & $C_{1}$ & 100 & 100 & 100 & 3 & 6 & 5.0800 & 9.7600 & 9 & 15 & 5.1276 & 9.4493 & 8.2388 & 19.6876 & 15.0707 & 71.1644\\

GEC1 & - & $C_{1}$ & 5 & 100 & 100 & 0 & 0 & 1.6800 & 0.8600 & 15 & 7 & 0.0584 & 0.0797 & 0.3425 & 0.2916 & 2.3575 & 1.1866\\

GEC2 & - & $C_{1}$ & 100 & 100 & 100 & 4 & 1 & 13.2200 & 1.0200 & 14 & 2 & 16.8354 & 5.4231 & 37.7360 & 6.8322 & 56.5505 & 9.9012\\

GEC3 & - & $C_{1}$ & 100 & 98 & 98 & 0 & 0 & 21.5306 & 10.4082 & 247 & 77 & 0.1724 & 0.2004 & 6.4364 & 3.5655 & 90.1907 & 29.5660\\

GRPY2 & - & $C_{1}$ & 100 & 100 & 100 & 0 & 0 & 9.6000 & 6.2800 & 200 & 100 & 1.4949 & 1.1804 & 20.1415 & 10.7672 & 390.2931 & 128.0368\\

Hil1 & 1 & $C_{1}$ & 50 & 100 & 100 & 0 & 0 & 3.7800 & 4.8800 & 65 & 20 & 0.1499 & 0.1732 & 0.8827 & 1.4954 & 11.9483 & 4.9273\\

Hil1 & 1 & $C_{2}$ & 10 & 80 & 100 & 0 & 0 & 6.5250 & 4.6000 & 18 & 24 & 0.0988 & 0.1252 & 1.1814 & 1.4468 & 3.3279 & 6.588\\

IKK1 & 2 & $C_{1}$ & 50 & 100 & 100 & 0 & 0 & 7.1400 & 7.2000 & 84 & 43 & 0.7771 & 1.0304 & 10.4211 & 12.6114 & 94.1988 & 52.9207\\

JOS1 & 1 & $C_{1}$ & 50 & 100 & 100 & 17 & 7 & 21.2800 & 11.0800 & 33 & 50 & 3.8329 & 1.7927 & 6.3339 & 3.2614 & 15.7729 & 10.2713\\

JOS1 & 1 & $C_{1}$ & 10 & 100 & 100 & 25 & 10 & 25.3200 & 11.2800 & 33 & 14 & 10.7817 & 7.6631 & 65.1590 & 27.5927 & 148.7070 & 92.1108 \\

KW2 & 1 & $C_{1}$ & 50 & 100 & 100 & 0 & 0 & 10.7200 & 11.7600 & 72 & 40 & 0.1492 & 0.1710 & 2.2122 & 2.9078 & 13.4446 & 9.7090\\

Lov5 & 1 & $C_{1}$ & 50 & 100 & 100 & 0 & 0 & 18.5200 & 10.3200 & 91 & 51 & 0.1480 & 0.1753 & 3.8904 & 2.4130 & 16.8968 & 10.5893 \\

MOP1 & 1 & $C_{1}$ & 50 & 100 & 100 & 1 & 1 & 1.5600 & 2.0000 & 4 & 5 & 1.0006 & 1.1697 & 1.5800 & 2.3998 & 3.5178 & 6.9476\\

MOP7 & 2 & $C_{1}$ & 50 & 100 & 100 & 34 & 14 & 253.4200 & 124.3400 & 435 & 216 & 63.8696 & 16.9014 & 487.7771 & 211.9162 & 1375.8300 & 611.1069\\

PNR & 1 & $C_{1}$ & 50 & 100 & 100 & 1 & 1 & 4.9000 & 12.5400 & 22 & 36 & 0.3173 & 0.3986 & 1.1102 & 3.8153 & 5.4461 & 17.4612\\

SD & 1 & $C_{1}$ & 50 &  100 & 100 & 4 & 4 & 18.0000 & 8.5600 & 43 & 18 & 9.1147 & 2.8166 & 38.2251 & 5.9329 & 89.7031 & 15.6642\\

Toi4 & 1 & $C_{1}$ & 50 & 100 & 100 & 2 & 2 & 18.5800 & 17.0200 & 139 & 113 & 0.4780 & 0.5822 & 3.9861 & 4.1424 & 26.0522 & 29.6788 \\

TRIDIA & 2 & $C_{1}$ & 50 & 100 & 100 & 0 & 0 & 27.5200 & 20.6200 & 426 & 252 & 0.9112 & 1.1697 & 41.0800 & 40.1414 & 606.2121 & 496.6671\\

VU2 & 1 & $C_{1}$ & 10 & 100 & 100 & 14 & 1 & 23.8000 & 7.8600 & 37 & 17 & 2.2235 & 0.2687 & 4.2168 & 1.7040 & 7.4383 & 4.6867\\

VU2 & 1 & $C_{2}$ & 10 & 100 & 100 & 5 & 1 & 22.3400 & 1.8000 & 26 & 10 & 1.8430 & 0.6982 & 7.1037 & 2.7312 & 13.2969 & 7.5321 \\

\bottomrule 
\end{tabular}}
\caption{Performance metrics of the proposed methods on the test problems for a hundred arbitrarily chosen initial points}
\label{performance_table} 
\end{center}
\end{table}
\end{landscape}

\begin{figure} 
  \centering
  \subfloat[Solver]
  {\includegraphics[width=0.25\textwidth]{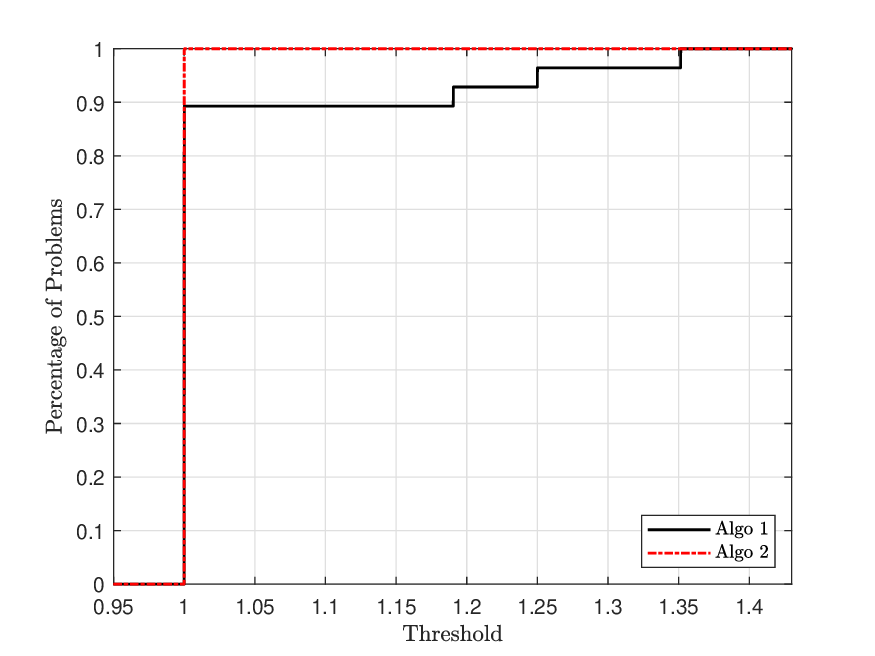}}
  \hfill
  \subfloat[Min iteration count]
 {\includegraphics[width=0.25\textwidth]{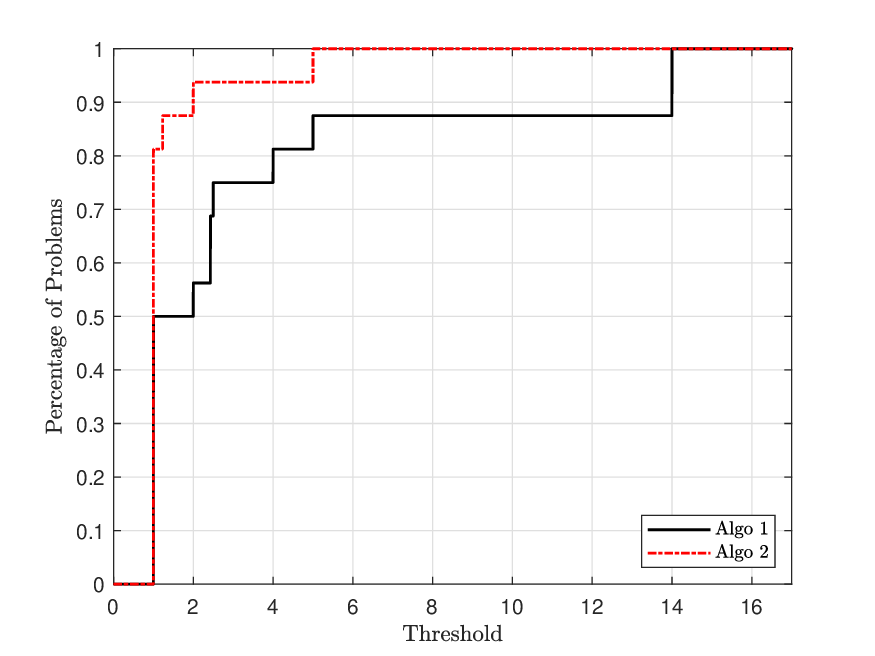}}
\hfill
  \subfloat[Mean iteration count]
{\includegraphics[width=0.25\textwidth]{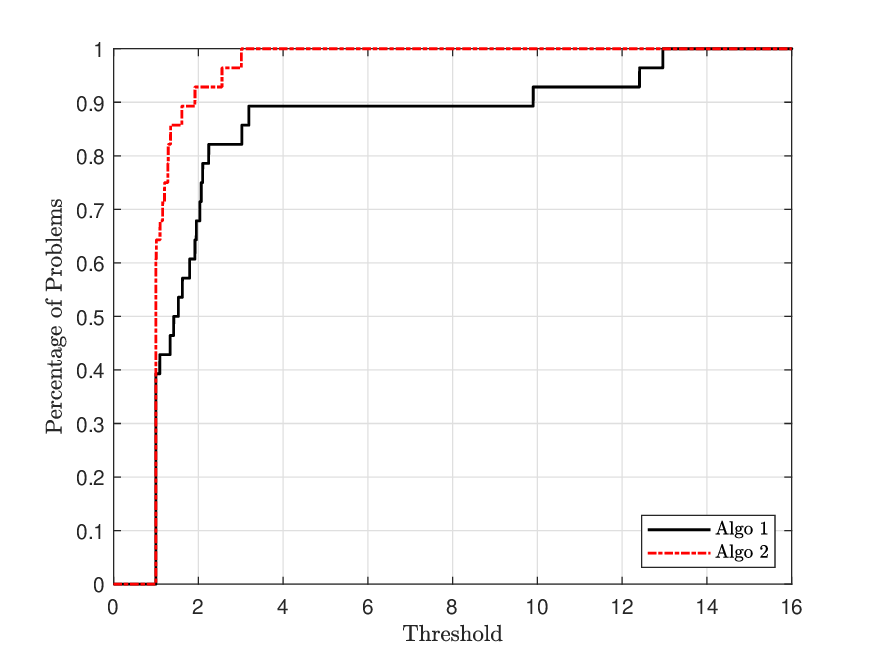}}
\hfill
  \subfloat[Max iteration count]
{\includegraphics[width=0.25\textwidth]{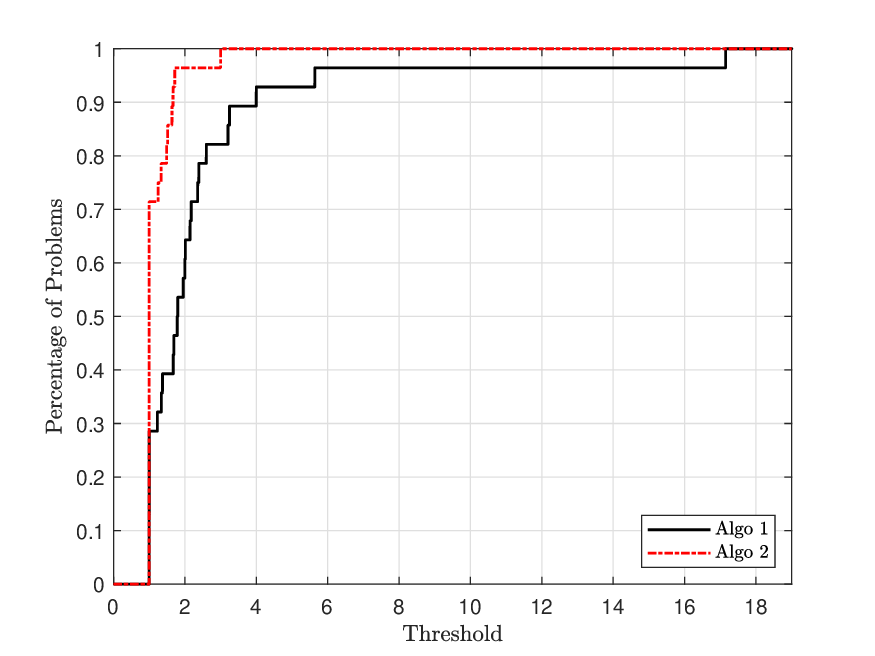}}
\hfill
  \subfloat[Min computational time]
{\includegraphics[width=0.33\textwidth]{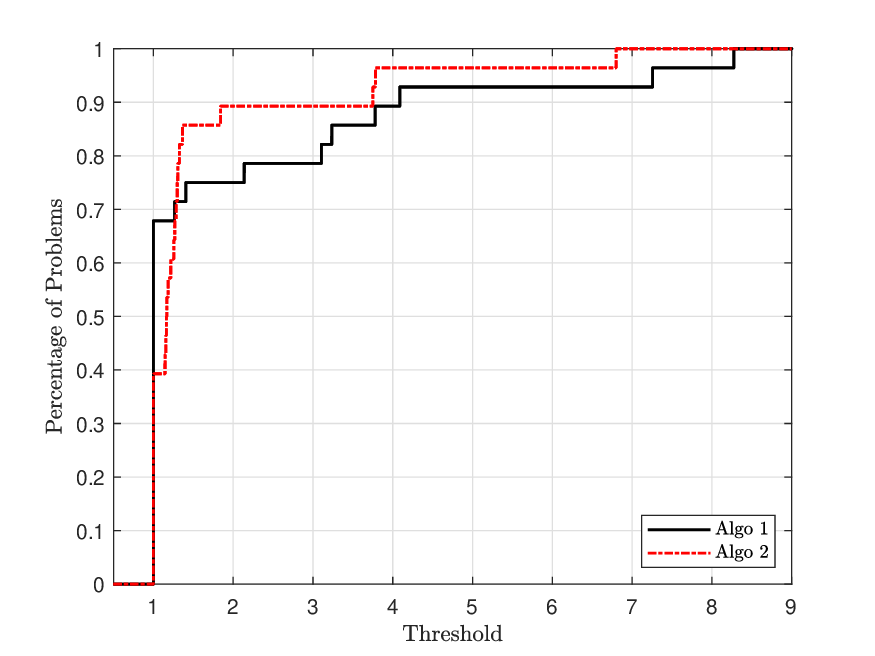}}
\hfill
\subfloat[Mean computational time]
{\includegraphics[width=0.33\textwidth]{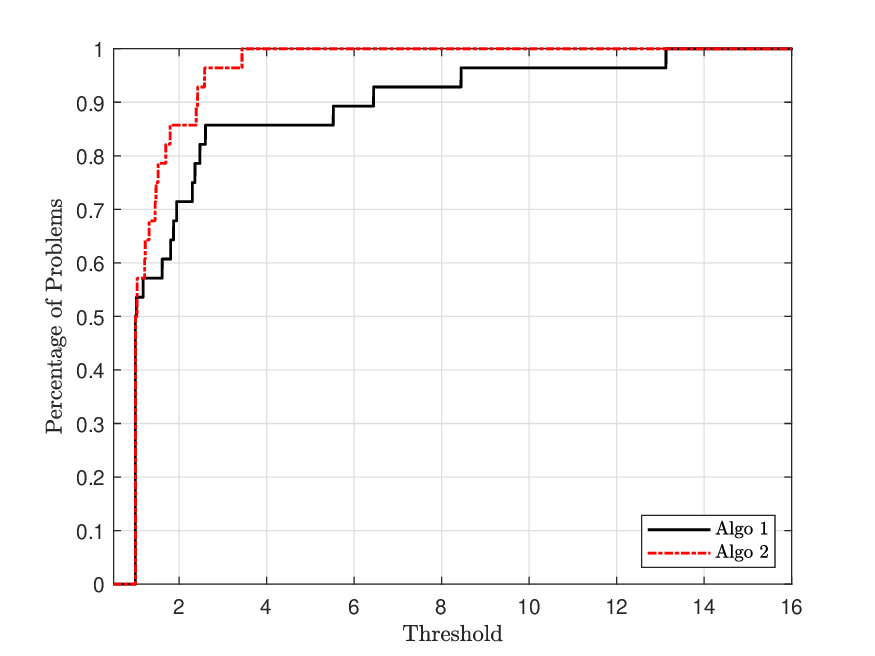}}
\hfill
\subfloat[Max computational time]
{\includegraphics[width=0.33\textwidth]{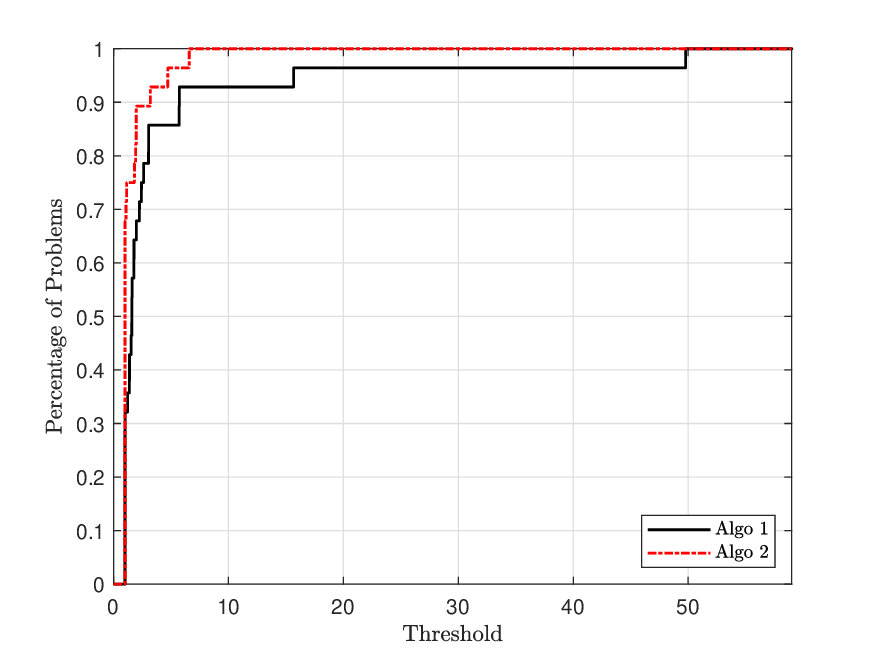}}
\caption{Performance profiles for Algo \ref{algo 1} and Algo \ref{algo} for the test problems given in Table \ref{Test_Problems}}\label{performance_profiles_graph}
\end{figure}

In view of the performance of the proposed methods under the metrics---solve, iteration count, and computational time---over the considered problems, the following observations can be made.
\begin{itemize}
    \item From the performance profile shown in Fig.~\ref{performance_profiles_graph}, it is evident that Algo~\ref{algo} outperforms Algo~\ref{algo 1}.

    \item There are some problems for which the performance of Algo~\ref{algo 1} is better than that of Algo~\ref{algo}, as can be clearly observed in Table~\ref{performance_table}.

    \item We observed during the computations that Algo~\ref{algo} generally required more computational time per iteration than Algo~\ref{algo 1}, particularly in cases where the number of variables involved in solving problem~\eqref{sub_prob_a} was large.

    \item It is noted that the step-size generated by Algo~\ref{algo 1} occasionally decreased below $10^{-15}$. Consequently, a desirable solution could not be obtained.

\end{itemize}

In Fig. \ref{Grap}, the bunches of points on the left and right sides represent $\{H(x^{k})\}$ and $\{f^{j}(x^{k})\}_{j\in [p]}$, respectively.  
The initial set is marked in red, the intermediate sets are shown in blue, and the terminal set is in green.

\begin{figure}[]
  \centering
  \subfloat[GRPY2 with $C_{1}$ by Algo \ref{algo}]{\includegraphics[width=0.5\textwidth]{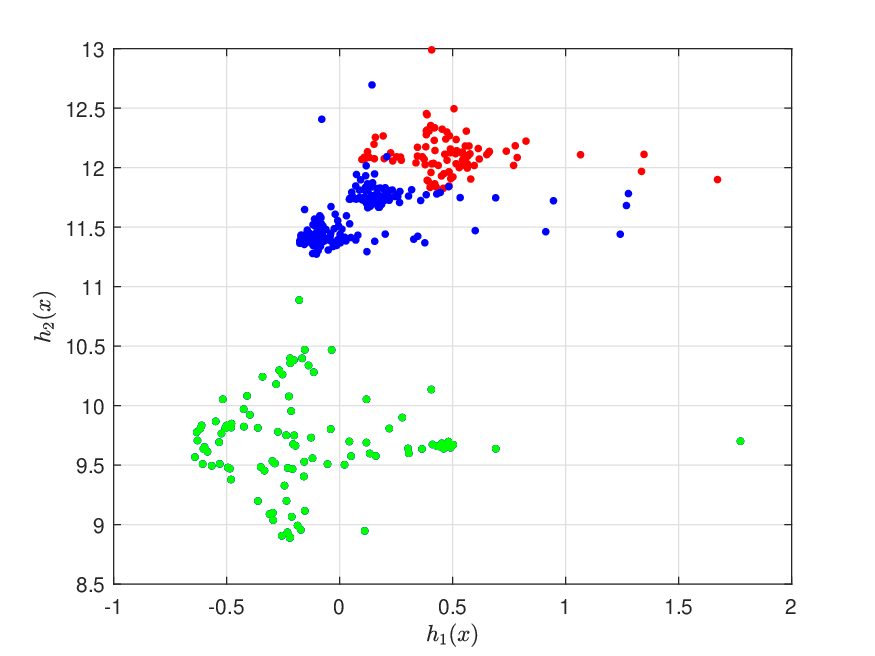}}
  \hfill
  \subfloat[After removing uncertainty from (a)]{\includegraphics[width=0.5\textwidth]{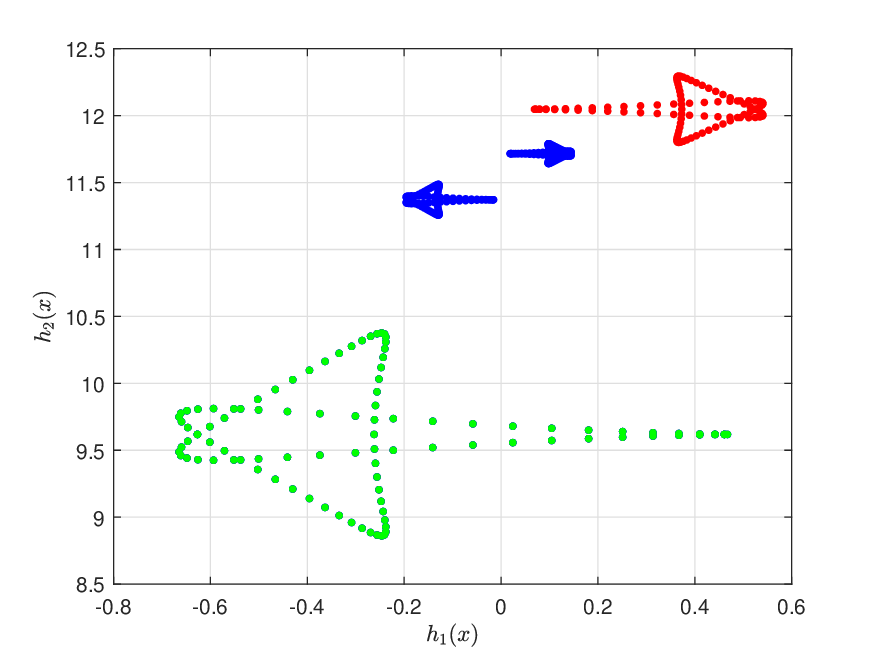}}
    \hfill
  \subfloat[IKK1 with $C_1$ by Algo \ref{algo 1}]{\includegraphics[width=0.5\textwidth]{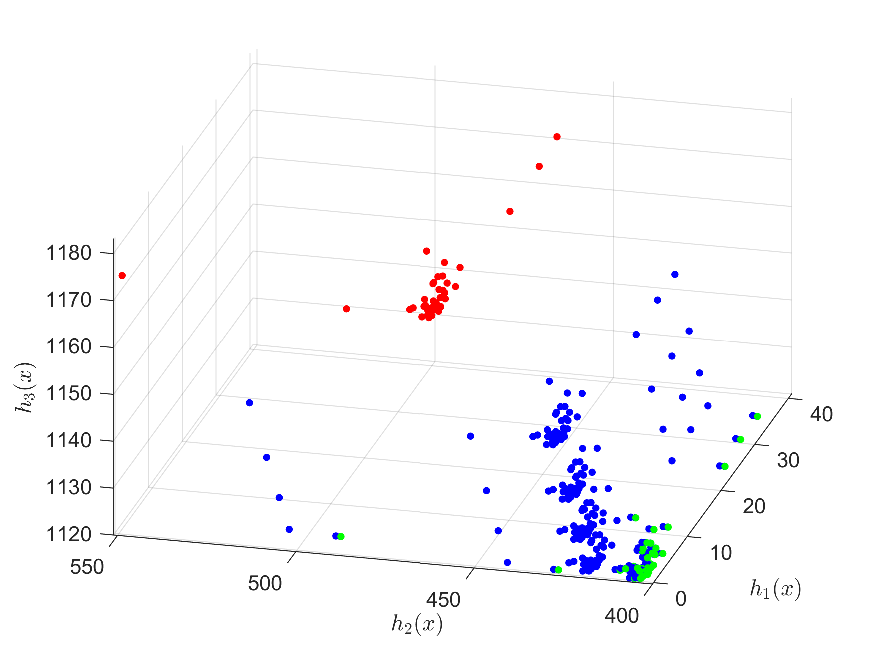}}
  \hfill
  \subfloat[After removing uncertainty from (c)]{\includegraphics[width=0.5\textwidth]{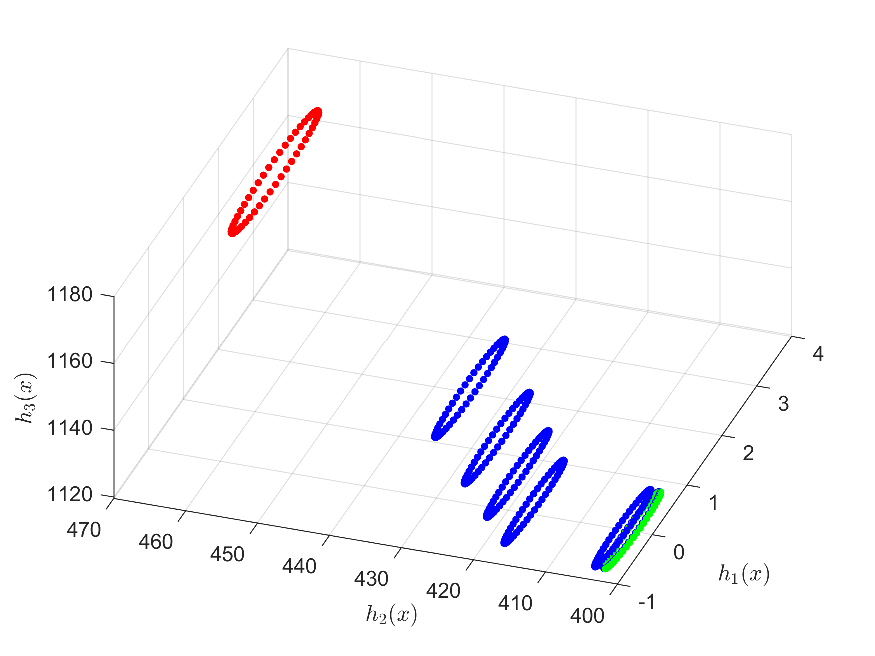}}
    \hfill
  \subfloat[SD with $C_1$ by Algo \ref{algo}]{\includegraphics[width=0.5\textwidth]{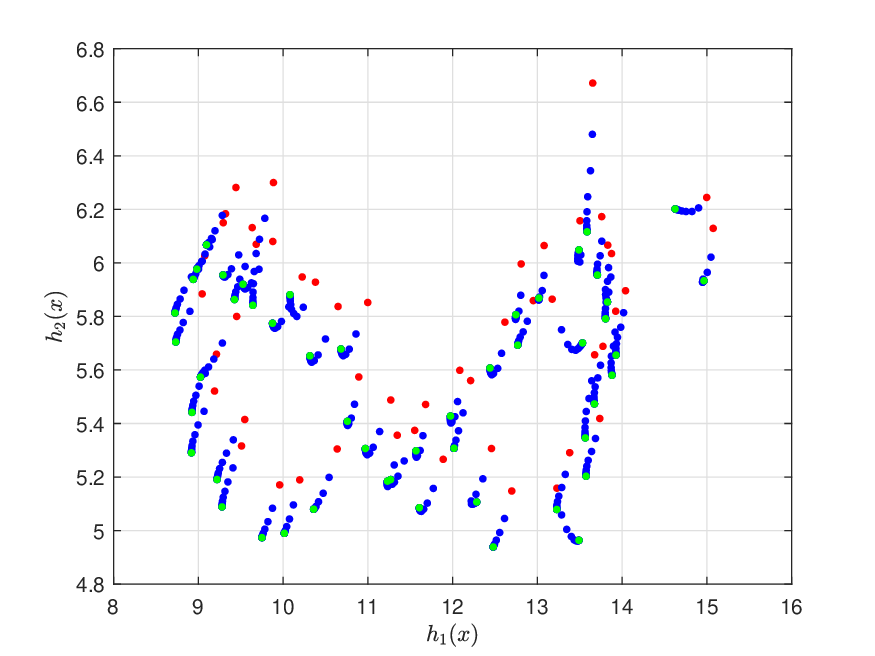}}
    \hfill
  \subfloat[After removing uncertainty from (e)]{\includegraphics[width=0.5\textwidth]{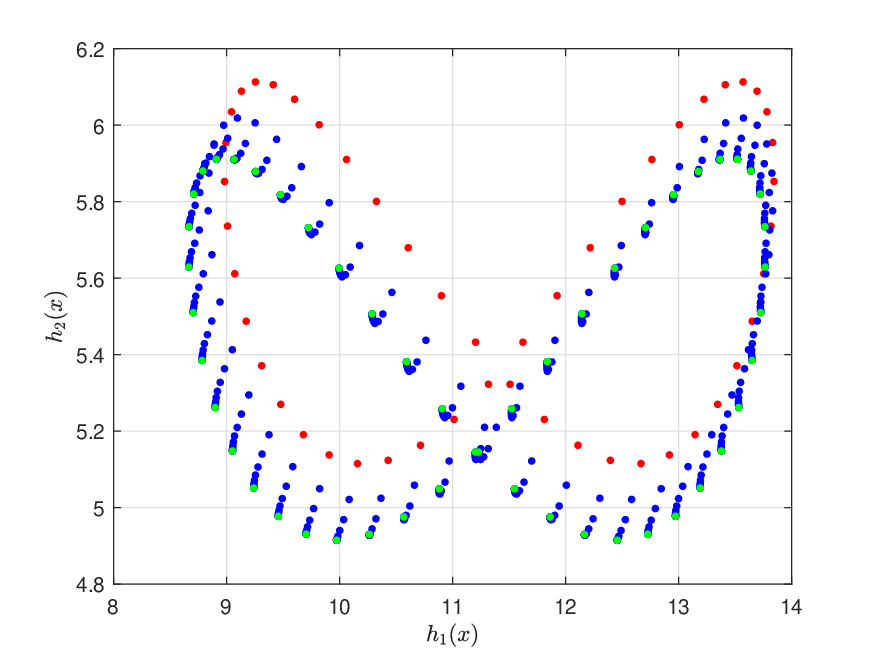}}
  
  \caption{The left side figures show the movement of the objective map of the problems corresponding to the iterative point from an arbitrarily chosen initial point to an optimal point in the image space, and the right side figures show the movement of the corresponding set-valued map without uncertainty ($g^{j}=0,~j\in [p]$).}{\label{Grap}}
\end{figure}

\newpage 
\section{Conclusion and future scope}\label{sec.7}
In this work, we have proposed proximal gradient methods (Algorithms \ref{algo 1} and \ref{algo}), with and without line search, for solving the set-valued optimization problem~\eqref{sop}.  
For the variant without line search, an additional assumption was imposed, that is, the gradient of the continuously differentiable functions, $ f^{j}$ for $j\in [p]$, involved in the objective map $H$ is Lipschitz continuous.

In order to propose the methods, a descent lemma (Lemma \ref{descent lemma}) in the vector case and a stationarity condition \eqref{efficient_stationary} for the vector problem~\eqref{VOP} were established.  
Furthermore, the notions of the minimal set, and the partition set of the objective map at a given point were introduced in \eqref{minimal_set} and \eqref{pat_set}, respectively.  
Based on the notion partition set \eqref{pat_set} of the objective map at a given point, an approach, mentioned in Remark \ref{N_S_c_optimal}, was developed to identify whether a given point is a weakly minimal point of \eqref{sop} by employing a family of vector optimization problems at that point. 
Subsequently, the notion of a stationary point (Definition \ref{D_Stat}) of \eqref{sop} was defined by employing the stationarity condition in the vector case, and an equivalent characterization was provided in Theorem \ref{equ_stationary}.  
Thereafter, the concept of a $C$-descent direction (Definition \ref{C-desent_def}) for a non-stationary point of the problem~\eqref{sop} was examined; in particular, it was shown that a solution $\bar{v}_{l}$ of~\eqref{opt_sol} at a point $\bar{y}$ also served as a descent direction. 
In the context of the proximal gradient method with line search, an Armijo-type line search condition, given in \eqref{Armijo_LS}, was formulated, and the existence of a step-size $\alpha_{k}$ was established for a non-stationary point $x^{k}$ of \eqref{sop} along a descent direction $v^{k}_{l}$ for any $l>0$ in Theorem \ref{LS_Theorem}(i).  
Accordingly, in Theorem \ref{LS_Theorem}(ii), it was demonstrated that the new point $x^{k}+\alpha_{k}v^{k}_{l}$, obtained through this line search condition, satisfied monotonicity in the descent sense.  
Similarly, in the case of without line search, at the $k^{\text{th}}$ iteration, the step-size $\alpha_{k}$ was fixed at one, while the parameter $l$ was adjusted such that $l>\tfrac{L}{2}$.  
It was shown in Proposition \ref{w_line_mono} that, for any iterate $x^{k}$, the new point $x^{k}+v^{k}_{l}$ with $l>\tfrac{L}{2}$ satisfied monotonicity.  
Accordingly, both proposed methods were confirmed as descent methods.

We have discussed the well-definedness of the proposed methods and proved their global convergence by Theorems \ref{Con_algo1} and \ref{con_alg2}. Moreover, Theorem \ref{Con_algo1} guaranteed that the sequence of non-stationary points generated by the steepest descent method, proposed in \cite{steepest2021set_optimization}, converges to a stationary point even in the absence of regularity as discussed in Remark \ref{con_steep_dest}. 
Furthermore, by imposing the assumption that $\nabla f^{j}$, $j\in [p]$, is Lipschitz continuous, we derived a minimum step-size that Algorithm~\ref{algo 1} can generate in Lemma \ref{lemma_step_size}, and based on this, we demonstrated in Theorem \ref{T_com_ana_1} that the computational complexity of Algorithm~\ref{algo 1}.  
With the condition $l \geq L$, we have also analyzed the complexity of Algorithm~\ref{algo} in Theorem \ref{T_com_ana_2}. 
Based on this computational analysis, we obtained the convergence rate $\mathcal{O}(1/\sqrt{k})$ for the proposed methods.  
Finally, a numerical comparison was carried out on several problems involving uncertainty to assess their practical performance. 
The results showed in Fig. \ref{performance_profiles_graph} that Algorithm~\ref{algo} outperformed Algorithm~\ref{algo 1} on the considered problems. 
However, it is worth noting that Algorithm~\ref{algo 1} requires a weaker condition than Algorithm~\ref{algo}. 
Thus, both proposed algorithms possess their own advantages.

In this work, we studied the convergence rate and computational complexity for the general problem \eqref{sop}.  
A natural extension is to carry out these analyses for special cases of \eqref{sop}, such as convex and strongly convex problems, where more precise results can be obtained.  
During the computational experiments, it was observed that in order to preserve the descent property, the step length occasionally dropped below $10^{-15}$.  
As a consequence, Algorithm \ref{algo 1} stagnated and failed to reach the desired precision.  
This observation motivates the development of a proximal gradient method equipped with non-monotone line search conditions.  
Another possible direction is to set up \eqref{sop} with $\ell_{1}$-norm type functions and study applications of this framework in image processing.

\subsubsection*{Funding}
Debdas Ghosh is thankful for the financial support from the Core Research Grant (CRG/2022/001347) from SERB, India.  Ravi Raushan thankfully acknowledges financial support from CSIR, India, through a research fellowship (File No. 09/1217(13822)/2022-EMR-I) to carry out this research work.

\subsubsection*{Data availability}
There is no data associated with this paper.

\subsection*{Declarations}    

\subsubsection*{Funding and/or Conflicts of interests/Competing interests} 
The authors do not have any conflicts of interest to declare. They also do not have any funding conflicts to declare.

\subsubsection*{Ethics approval} 
This article does not involve any human and/or animal studies.

\end{document}